\documentclass[11pt]{amsart}

\usepackage{graphicx} 
\usepackage[T1]{fontenc}
\usepackage[paper=a4paper, left=34mm, right=34mm, top=41mm, bottom=39mm]{geometry}
\usepackage[foot]{amsaddr}

\usepackage[dvipsnames]{xcolor}
\usepackage[utf8]{inputenc}
\usepackage{amssymb}
\usepackage[centertags]{amsmath}
\usepackage{amscd}
\usepackage{amsthm}
\usepackage{enumerate}
\usepackage{mathtools}
\usepackage{mathrsfs}
\usepackage{tikz}
\usepackage{wrapfig}
\usepackage{float}
\usepackage{enumitem}
\usepackage{csquotes}

\usepackage{multirow}
\usepackage{amsmath}
\usepackage{booktabs}
\usepackage[super]{nth}
\usepackage[english]{babel}
\usepackage{accents}
\usepackage{enumitem}
\usepackage{xfrac}
\usepackage[colorlinks=true,linkcolor=purple,citecolor=ForestGreen]{hyperref}

\usepackage{epsfig}
\usepackage{color}
\usepackage{indentfirst}
\graphicspath{ {Images/} }
\usepackage{afterpage}
\usepackage{comment}
\usepackage[footnotesize]{caption}
\DeclarePairedDelimiter{\abs}{\lvert}{\rvert}

\DeclareMathSymbol{\widetildesym}{\mathord}{largesymbols}{"65}

\newcommand{\N}{\mathbb{N}}

\newcommand{\E}{\mathbb{E}}
\newcommand{\Z}{Z_{n,[w]}}

\DeclareMathOperator{\arccosh}{cosh^{-1}}
\newcommand{\norm}[1]{\left\lVert#1\right\rVert}

\makeatletter
\newtheorem*{rep@theorem}{\rep@title}
\newcommand{\newreptheorem}[2]{%
\newenvironment{rep#1}[1]{%
 \def\rep@title{#2 \ref{##1}}%
 \begin{rep@theorem}}%
 {\end{rep@theorem}}}
\makeatother

\makeatletter
\def\blfootnote{\xdef\@thefnmark{}\@footnotetext}
\makeatother

\newtheorem{theorem}{Theorem}[section]
\newtheorem{defi}{Definition}
\newtheorem{prop}[theorem]{Proposition}
\newtheorem{lemma}[theorem]{Lemma}
\newtheorem{corol}[theorem]{Corollary}
\newtheorem{remark}[defi]{Remark}

\newreptheorem{theorem}{Theorem}
\newreptheorem{prop}{Proposition}

\usepackage{thmtools}
\usepackage{cleveref}
\usepackage{thm-restate}
\usepackage{algorithm}
\usepackage[noend]{algpseudocode}

\makeatletter
\def\BState{\State\hskip-\ALG@thistlm}
\makeatother

\title{The systole of random hyperbolic 3-manifolds}
\author{Anna Roig-Sanchis}
\address{Laboratoire J.A. Dieudonné, Université Côte d'Azur, CNRS, 06108 Nice}
\email{\href{mailto:anna.roig-sanchis@univ-cotedazur.fr}{anna.roig-sanchis@univ-cotedazur.fr}}
\date{\today}

\begin{document}

\begin{abstract}
We study the systole of a model of random hyperbolic 3-manifolds introduced in \cite{Bram_Jean}, answering a question posed in that same article. 
These are compact manifolds with boundary constructed by randomly gluing truncated tetrahedra along their faces. We prove that the limit, as the volume tends to infinity, of the expected value of their systole exists and we give a closed formula of it. Moreover, we compute a numerical approximation of this value.
\end{abstract}


\maketitle
\section{Introduction}

The \textit{systole} of a hyperbolic $n$-manifold $M$ is the length of the shortest geodesic of $M$. It constitutes one of the simplest geometric invariants of M, yet gives very rich information about the manifold. Indeed, it is related to many other of its geometric properties, like the volume \cite{Gromov}, the diameter \cite{Bavard, BalacheffParlier}, the kissing number \cite{Parliersys, FourtierBourquePetri}, or the Cheeger constant \cite{BrooksCheegercst}. 


The study of systolic geometry started early in 1949 with Loewner, who proved an inequality for the systole of the 2-torus in terms of its area. This was first mentioned in Pu's paper \cite{Pu}, who then stated in this same work a similar inequality for the real projective plane. Some years later, Gromov gave one of the most well known results in the area \cite{Gromov}; he proved an upper-bound on the 1-systole of essential $n$-manifolds (such as hyperbolic manifolds) in terms only of their volume, for all $n\geq2.$

Over the years, there has been extensive research on the behavior of the systole of hyperbolic manifolds. It is known, for instance, that the infimum of the systole of any closed hyperbolic $n$-manifold with vol($M$)$\leq v$ is zero for $v$ large enough if $n=2,3$, and tends to zero as $v\rightarrow \infty$ for $n\geq4$ \cite{Agol, Belolipetsky}. Nonetheless, the behaviour of the maximum of this value as the volume grow, is still unknown for any $n\geq2$. Thus, many results in the field are about finding sharper lower and upper bounds of it, and examples of manifolds with large systole. (see Section \ref{NotesRef} for more details). 

Another natural question to ask is what the systole of a typical manifold is. A way to approach  this question by considering random constructions and studying their properties.
In this article, we study the asymptotic behaviour of the systole of random hyperbolic 3-manifolds $M_n$ built using the model of random triangulations, answering a question posed in \cite[Question 3]{Bram_Jean}. 
The answer is given by the following results: 
\begin{theorem} \label{sysMn}
Let $\{l_i\}_{i\geq1}$ be the ordered set of all possible translation lengths coming from (classes of) words $[w] \in \mathcal{W}$. Then,
\[ \lim_{n\rightarrow \infty} \E[\text{sys}(M_n)] = \sum_{i= 1}^{\infty} \bigg( \prod_{\scriptscriptstyle [w]\in \mathcal{W}_{l_{i-1}}} \exp \bigg(\frac{\abs{[w]}}{2\abs{w}3^{\abs{w}}} \bigg)  \bigg) \bigg(1 - \prod_{\scriptscriptstyle [w]\in \mathcal{W}_{l_i} \setminus \mathcal{W}_{l_{i-1}}} \exp\bigg(\frac{\abs{[w]}}{2\abs{w}3^{\abs{w}}} \bigg)  \bigg)\cdot l_i.\]
\end{theorem}
Here $\mathcal{W}$ corresponds to a collection of matrices in $\text{SL}(2, \mathbb{Z}[i])$. The precise definition can be found in Section \ref{defW}.
This theorem proves the existence of the limit, as $n\rightarrow\infty$, of the expected value of the systole of a model of compact hyperbolic 3-manifolds with boundary, by giving an explicit (and computable) expression of it. The second main result, Proposition \ref{sysvalue}, completes the answer to the question by computing a sharp numerical approximation of this value.

\begin{prop} \label{sysvalue}
We have: 
\[ \lim_{n \rightarrow \infty} \mathbb{E}(\text{sys}(M_n)) = 2.5603331268388752 \pm 2.95489 \cdot 10^{-16}. \]
\end{prop}

\subsection{Structure of the proofs}

The strategy of the proof of Theorem \ref{sysMn} consists in the following two steps. 

First, we compute the limit of the expected value for a model of hyperbolic manifolds $Y_n$. These are non-compact hyperbolic manifolds made out of a gluing of hyperbolic ideal right-angled regular octahedra. The relation between them is that the former manifolds $M_n$ can be seen as the Dehn filling of these $Y_n$. We do that because the geometry of these $Y_n$ is better understood -by construction of the manifolds and their hyperbolic metric- than the one of $M_n$, so it is easier to study their geometric properties. 

Thus, to get here the expression of the expected value, we translate the study of the number of closed geodesics of certain lengths in $Y_n$ to the study of their corresponding cycles in the dual graph $G_{Y_n}$, a random 4-regular graph. Then, we rely on our Theorem \ref{cyclesZn} from \cite{me}, which gives the asymptotic behaviour of these random variables counting the cycles in the graph. The proof of this step requires of the convergence of the infinite sum given by the expected value. For that, we apply the dominated convergence theorem, using mainly graph theory tools such as Corollary \ref{graph}, from McKay-Wormald-Wysocka \cite{Mckay-Wormald}.

The second part of the proof is then to see that the result applies also to the compactified manifolds $M_n$. For that, we prove that the contribution to the expected value of a set of "bad" manifolds -in which the systole could degenerate- goes to zero as $n\rightarrow\infty$. The study of these potential "bad" manifolds relies heavily on the proof of \cite[Proposition 4.1]{me}, and uses results coming from both graph theory, like Corollary \ref{graph} or a result from Bordenave \cite{Tangle-free}, and hyperbolic geometry, like \cite[Theorem 9.30]{Futer_Purcell} from Futer-Purcell-Schleimer. 

On the other hand, the argument for finding the numerical value of Proposition \ref{sysvalue} goes also in two steps. Indeed, we divide the infinite sum corresponding to the expected value in two terms: a computable part and an error term, that we bound. Thus, the first sum is computed with a Sage program, showed later in the text. For the error term, we carefully study the probabilities appearing, and use Theorem \ref{cyclesZn} together with some computational data to obtain the bound. The main difficulty that we encounter - which differentiates it from the two dimensional case \cite{Bram1}- is that there is not a natural way of ordering the lengths $l_i$ from the information given by the cycles. This makes the computation much less straightforward. 


\subsection{Organization}
In Section \ref{model}, we recall the probabilistic model of random 3-manifolds introduced in \cite{Bram_Jean}. We also explain briefly the relation between the geodesics in the manifold $Y_n$ and the cycles in its dual graph, and state a result about the asymptotic distribution of these cycles. We end the section by giving an expression for the systole of this related manifold, using the previous result. Section \ref{convergence} is then devoted to justify that the expression of the systole of Section \ref{model} is indeed valid. Thus, it contains the argument needed to apply the dominated convergence theorem. Then, in Section \ref{systoleforMn}, we show that this expression can be transferred to the manifolds $M_n$ we're interested in. The convergence argument of section \ref{convergence} together with this gives the proof of Theorem \ref{sysMn}. Finally, in Section \ref{numericalvalue}, we compute the numerical value of the expected value of the systole, obtaining Proposition \ref{sysvalue}.


\subsection{Notes and references} \label{NotesRef}
The search for lower and upper bounds for the maximal possible value of the systole of hyperbolic manifolds, as the complexity of these grow, have been a source of interest for a long time. There is a general upper bound for all dimensions $n\geq2$ coming from a simple volume growth argument \cite{Buserbook}.
For $n\geq4$, this remains the best upper bound currently known.
For surfaces, a better one coming from the injectivity radius was found by Bavard \cite{Bavard}, which was improved in later work by Fourtier-Bourque-Petri \cite{FourtierBourquePetri}. In the case of hyperbolic 3-manifolds, this bound has also been improved recently by \cite{BonifacioMazacPal}. In all cases, the bound has logarithmic growth.

As for lower-bounds, Brooks \cite{Brookssys} and Buser-Sarnak \cite{BuserSarnak} prove the existence of sequences of closed hyperbolic surfaces $(S_k)_k$ with genus increasing in $k$ and sys$(S_k) \geq \frac{4}{3}\log(g_k) + O(1)$, using arithmetic constructions. This was generalised later on by Katz-Schapps-Vishne \cite{Katz} for other classes of surfaces and 3-manifolds, and by Murillo \cite{Murillo} for $n$-manifolds, with similar methods. In dimension 2 there are also combinatorial constructions which provide examples of surfaces with logarithmic systoles; these are given by Petri-Walker \cite{BramWalker} and Liu-Petri \cite{LiuPetri}, the latter using probabilistic techniques.

Another interesting question to ask is which manifolds attain the maximal values for the systole. The answer to this is only known for two dimensions and in the case of genus 2, where Jenni \cite{Jenni} showed that the maximiser for the systole is the Bolza surface.

Finally, the systole has also been studied in other models of random 3-manifolds. Feller-Sisto-Viaggi \cite{FellerDiam} found an upper bound to the decay rate of the length of the shortest geodesic of a 3-manifold under the model of random Heegaard splittings. A similar result was proven for the model of random mapping tori by Taylor-Sisto \cite{TaylorSisto}.

\subsection{Acknowledgements}
I would like to thank first my PhD advisor Bram Petri, for his help and time spent discussing all the details of this article. I would also like to thank the anonymous referee for very helpful comments. Finally, I want to thank my lab colleagues Pietro Mesquita Piccione and Luis Cardoso for very useful discussions.

\section{The systole: an expression in terms of cycles} \label{model}

\subsection{The probabilistic model} \label{probmodel}
We start by briefly recalling the model of random hyperbolic 3-manifolds used for our results in this paper. A more detailed explanation can be found in \cite{me}. This probabilistic model, called \textit{random triangulations}, is an analogue in three dimensions of Brooks and Makover's model for random surfaces \cite{BrooksMakover}. 

Thus, a random three-manifold $N_n$ is constructed by gluing together $n$ truncated tetrahedra along their hexagonal faces (see Figure \ref{tetra}). The randomness of the model comes from the gluing: both the partition of the $4n$ faces into pairs, and the cyclic-order-reversing gluings in which these pairs of faces are glued, is chosen uniformly at random. This complex turns into an oriented three-manifold with boundary.

\begin{figure}[H]
    \centering
    \includegraphics[scale=0.4]{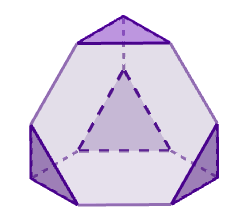}
    \caption{A truncated tetrahedron.}
    \label{tetra}
\end{figure}

One can consider the dual graph of this complex $N_n$. This is the graph obtained by putting a vertex in each truncated tetrahedron of $N_n$ and joining these vertices with an edge whenever they have a face in common. Since each truncated tetrahedron is glued to another or to itself along its four hexagonal faces, the resulting dual graph is a random 4-regular graph.

In this article, we will consider manifolds with simple dual graph. These will be denoted by $M_n$. The reason why this doesn't affect our results is because any property P that holds a.a.s. for $N_n$ (i.e, $\mathbb{P}[ N_n \ \textrm{has} \ P] \rightarrow 1$ as $n\rightarrow \infty$), also does for the manifold conditioned on not having loops or multiple edges in its dual graph \cite{Moments}.

Several geometric and topological properties of these manifolds $M_n$ are proven in \cite{Bram_Jean}. The most important one for our purpose is the following: 
\begin{theorem}[\cite{Bram_Jean}, Theorem 2.1] \label{hyperbolic}
    \[\lim_{n\rightarrow \infty} \mathbb{P}[M_n \ \textrm{carries a hyperbolic metric with totally geodesic boundary}] = 1.\]
\end{theorem} 

The proof of hyperbolicity passes through the construction of another model of hyperbolic manifolds made of a gluing of ideal right-angled hyperbolic regular octahedra, denoted by $Y_n$. It is seen later in \cite{Bram_Jean}, that the $M_n$ are homeomorphic to the Dehn filling of these $Y_n$. 

The latter play an important role in the study of the geometry of $M_n$. Indeed, since their hyperbolic structure is more understood, a general strategy to prove properties of $M_n$ is to prove them for $Y_n$ and then check that they hold after the compactification process. This was the line of argument followed in \cite{me} to study the length spectrum of $M_n$  -the (multi-)set of lengths of all its closed geodesics-, and the one we will use to study the systole.

\subsection{Curves and cycles}

In \cite[Section 3.2, 3.3]{me} we explain in detail how we extract information about the closed geodesics in $Y_n$ by looking at the dual graph. 

In summary, the complex $Y_n$ made of octahedra deformation retracts into its dual graph $G_{Y_n}$, which is a random 4-regular graph. This, enriched with some extra information on the edges regarding the orientation-reversing gluings of every pair of faces, encodes the combinatorics of the complex.

The closed curves $\gamma$ in $Y_n$ can be homotoped to closed paths in the dual graph. As the expected number of non-simple closed paths of bounded length in $G_{Y_n}$ goes asymptotically to 0 as $n\rightarrow\infty$ \cite{Moments}, these $\gamma$ will be homotoped to simple closed paths, namely cycles. 

\begin{figure}[H]
    \centering
    \includegraphics[scale=0.3]{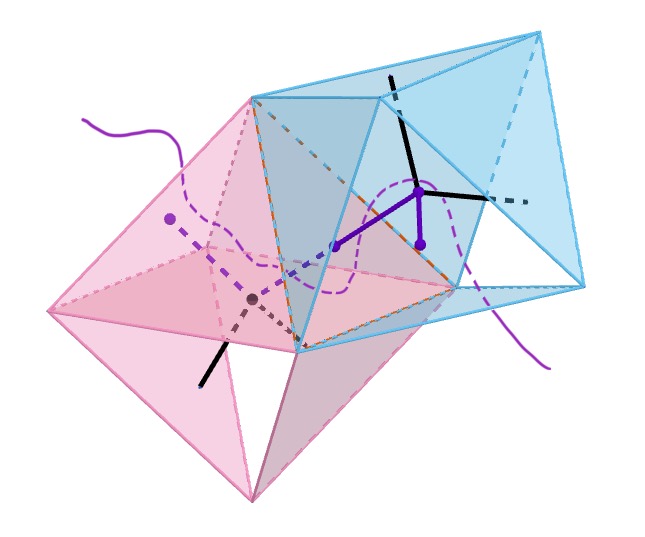}
    \caption{Homotopy of a curve into the dual graph.}
    \label{Homotopy_path}
\end{figure}

These cycles, in turn, can be described by (equivalence classes of) words $w$. These are products of the following elements in PSL(2, $\mathbb{Z}[i]$) $<$ PSL(2, $\mathbb{C}$):
\begin{equation} \label{matrices}
    S = \begin{pmatrix} 1 & 1 \\ 0 & 1 \end{pmatrix} \quad \quad R = \begin{pmatrix} -1 & i \\ i-1 & i \end{pmatrix} \quad \quad L = \begin{pmatrix} i & i \\ i+1 & 1 \end{pmatrix} \quad \quad \theta = \begin{pmatrix} 0 & i \\ i & 1 \end{pmatrix},
\end{equation}
describing the isometry realised by the mapping of each incoming and exiting face of the octahedra the curve traverses (see \cite{me} for more details). They have the form: 
\[w = w_{1} \Theta_1 \cdot w_{2} \Theta_2 \cdot \ldots \cdot w_{\abs{w}} \Theta_{\abs{w}},\]
where $w_{i} \in \{ S, R, L \}$ and $\Theta_i \in \{Id, \theta , \theta^2\}$ for $i=1, \ldots, \abs{w}$, and $\abs{w}$ is the length of the word, being the number of elements $w_i\Theta_i$ in $w$. Since we're working with manifolds $M_n$ with simple dual graph, all words that we consider will be of combinatorial length $|w|\geq3$.

Observe also that a curve can be described by different words, depending on the starting point, the direction and the orientation of the gluings of the octahedra. Thus, we define a notion of equivalence class of words in terms of these three parameters (a complete description can be found in \cite{me}). We will denote by W the set of all words $w$ with $|w|\geq 3$, and by $\mathcal{W}$ the set of all their equivalence classes $[w]$. 

Then, from these words, one can recover the length of the corresponding geodesics in the homotopy class of the curves $\gamma$ \cite[Chapter 5]{GeoTopoThurston}. Indeed, the relation between these is given by: 
\begin{equation} \label{l}
    l_{\gamma}(w) = 2\mathrm{Re}\bigg[\arccosh \bigg(\frac{\textrm{trace}([w])}{2}\bigg)\bigg].
\end{equation}

\subsection{The expected value of systole} \label{defW}

By definition of systole, if $M_n$ has systole $l>0$, it means that there is at least one closed geodesic of length $l>0$ in $M_n$, and no other closed geodesic smaller than this. Hence, in order to have insight on its expected value, we would need to know the number of geodesics of each possible length.

Using the relation (\ref{l}) above, on can try first to translate this counting to the counting of cycles in the dual graph $G_{Y_n}$ corresponding to each class of words $[w]$, as it was done for the study of the length spectrum in \cite{me}. For that, we define, in the probability space $(\Omega_n, \mathbb{P}_n)$ associated to the model $Y_n$, the random variable: 

\[Z_{n, [w]} : \Omega_n \rightarrow \N, \quad n\in \N, \ [w]\in \mathcal{W} \coloneqq W /\sim,\]
as
\[Z_{n, [w]} (\omega) \coloneqq  \#\{\textrm{cycles $\gamma$ on $G_{Y_n}$ : $\gamma$ is described by $[w]$}\}. \]

We proved in \cite{me} the following about the asymptotic distribution of these random variables.

\begin{theorem}[\cite{me}, Theorem 3.3] \label{cyclesZn}
Consider a finite set $\mathcal{S}$ of equivalence classes of words in W. Then, as $n \rightarrow \infty$,
\[\Z \rightarrow Z_{[w]} \quad \textrm{in distribution for all} \ [w]\in \mathcal{S},\]
where:
\begin{itemize}
    \item $Z_{[w]}: \N \rightarrow \N$ is a Poisson distributed random variable with mean $\lambda_{[w]}= \frac{\abs{[w]}}{3^{\abs{w}}2\abs{w}}$ for all $[w]\in \mathcal{S}$.
    \item The random variables $Z_{[w]}$ and $Z_{[w']}$ are independent for all $[w],[w'] \in \mathcal{S}$ with $[w] \neq [w']$.
\end{itemize}
\end{theorem} 

On the other hand, given $l>0$, let:
\[\mathcal{W}_{l} = \bigg\{[w]\in \mathcal{W} : \abs{w} > 2, \ \abs{tr([w])} > 2 \ \textrm{and} \ 2\mathrm{Re}\bigg[\arccosh \bigg(\frac{tr([w])}{2}\bigg)\bigg] \in [0,l]\bigg\}.\]

We can now write down an expression for the expected value of the systole of $Y_n$ in terms of these random variables as follows:  
\begin{align} \label{sysZn}
    \E[\text{sys}(Y_n)] &= \sum_{i= 1}^{\infty} \mathbb{P}\bigg[\begin{array}{c}
    \nexists \ \text{geodesics of length} \ l<l_i, \ \text{and} \\
    \exists \ \text{at least one geodesic of length } l_i \ \text{in} \ Y_n
    \end{array} \bigg] \ l_i \\ 
    &= \sum_{i= 1}^{\infty} \mathbb{P}\bigg[\begin{array}{c}
    \Z = 0 \ \text{for all} \ [w]\in \mathcal{W}_{l_{i-1}}, \ \text{and} \\
   \Z > 0 \ \text{for some} \ [w]\in \mathcal{W}_{l_i} \setminus \mathcal{W}_{l_{i-1}} 
    \end{array} \bigg] \ l_i,
\end{align}
where the sequence $\{l_i\}_{i\geq1}$ is the ordered set of all possible translation lengths coming from (classes of) words $[w] \in \mathcal{W}$, obtained using (\ref{l}), and $l_0 = 0$.

Thanks to Theorem \ref{cyclesZn}, we can compute the point-wise limits of these probabilities. Indeed, let $A_{n}^{i}$ denote the latter event, that is, for all $i\geq 1$,
\[A_{n}^{i} = \{\Z = 0 \ \text{for all} \ [w]\in \mathcal{W}_{l_{i-1}}, \ \text{and} \ \Z > 0 \ \text{for some} \ [w]\in \mathcal{W}_{l_i} \setminus \mathcal{W}_{l_{i-1}}\}.\] We have: 

\begin{prop} \label{limitZn}
Let $\{l_i\}_{i\geq1}$ be the ordered set of all possible translation lengths coming from (classes of) words $[w] \in \mathcal{W}$, and $l_0 = 0$. Then, for every $i\geq 1$, 
\begin{equation*}
    \lim_{n\rightarrow\infty} \mathbb{P}\big[ A_{n}^{i} \big] = \bigg( \prod_{[w]\in \mathcal{W}_{l_{i-1}}} \exp\bigg(\frac{\abs{[w]}}{2\abs{w}3^{\abs{w}}} \bigg)  \bigg) \bigg(1 - \prod_{[w]\in \mathcal{W}_{l_i} \setminus \mathcal{W}_{l_{i-1}}} \exp\bigg(\frac{\abs{[w]}}{2\abs{w}3^{\abs{w}}} \bigg)  \bigg).
\end{equation*}
\end{prop}

\begin{proof}
It was proven in \cite[Proposition 3.4]{me} that the number of classes of words $[w]\in \mathcal{W}_{l_i} \setminus \mathcal{W}_{l_{i-1}}$, for any $i\geq 1$, is finite. Hence, we can use Theorem \ref{cyclesZn} which, together with the independence of the random variables, gives:
\begin{align*} 
    \lim_{n\rightarrow\infty} \mathbb{P}\big[A_{n}^{i}\big] &= \bigg( \prod_{[w]\in \mathcal{W}_{l_{i-1}}} \mathbb{P}[Z_{[w]}=0]  \bigg) \bigg(1 - \prod_{[w]\in \mathcal{W}_{l_i} \setminus \mathcal{W}_{l_{i-1}}} \mathbb{P}[Z_{[w]}=0] \bigg) \\ 
    &= \bigg( \prod_{[w]\in \mathcal{W}_{l_{i-1}}} \exp\bigg(\frac{\abs{[w]}}{2\abs{w}3^{\abs{w}}} \bigg)  \bigg) \bigg(1 - \prod_{[w]\in \mathcal{W}_{l_i} \setminus \mathcal{W}_{l_{i-1}}} \exp\bigg(\frac{\abs{[w]}}{2\abs{w}3^{\abs{w}}} \bigg)  \bigg).  
\end{align*}
\end{proof}
Therefore, up to showing that we can switch the limit and the infinite sum in (\ref{sysZn}), this would allow us to obtain a precise computable expression of the expected value of the systole for the manifolds $Y_n$. Once we have this, the only thing remaining would be to see that this value doesn't degenerate for the compactified manifolds $M_n$.


\section{Convergence} \label{convergence}

Let us show that we can indeed swap the limit and the infinite sum appearing in the expression (\ref{sysZn}) of the expected systole, so that we can apply Proposition \ref{limitZn}, and obtain:

\begin{equation} \label{sysYn}
    \lim_{n\rightarrow \infty} \E[\text{sys}(Y_n)] = \sum_{i= 1}^{\infty} \bigg( \prod_{\scriptscriptstyle[w]\in \mathcal{W}_{l_{i-1}}} \exp \bigg(\frac{\abs{[w]}}{2\abs{w}3^{\abs{w}}} \bigg)  \bigg) \bigg(1 - \prod_{\scriptscriptstyle [w]\in \mathcal{W}_{l_i} \setminus \mathcal{W}_{l_{i-1}}} \exp\bigg(\frac{\abs{[w]}}{2\abs{w}3^{\abs{w}}} \bigg)  \bigg)\cdot l_i.
\end{equation}

For this, we will apply the dominated convergence theorem. 
A first naive observation is that: 
\[  \mathbb{P}\big[A_{n}^{i}\big] = \mathbb{P}\bigg[\begin{array}{c} \Z = 0 \ \text{for all} \ [w]\in \mathcal{W}_{l_{i-1}}, \ \text{and} \\ \Z > 0 \ \text{for some} \ [w]\in \mathcal{W}_{l_i} \setminus \mathcal{W}_{l_{i-1}} \ \end{array} \bigg] \leq \mathbb{P}[\Z = 0 \ \text{for all} \ [w]\in \mathcal{W}_{l_{i-1}}]. \]

We look then for a uniform upper bound for the latter. By definition of the random variable $\Z$, this tells us that there are no words whose length of the corresponding geodesic is in $[0,l_{i-1}]$. This condition implies, in turn, a lower-bound on the combinatorial length of the word. 
Indeed, if $w=M_1 \cdots M_k$, where $M_i \in \{S, R, L, S\theta, R\theta, L\theta, S\theta^2, R\theta^2, L\theta^2\}$, using the sub-multiplicity property of the operator norm, we have:
\begin{align*}
    \norm{w}_{\infty} = \norm{M_1 \cdots M_k}_{\infty} \leq \norm{M_1}_{\infty} \cdots \norm{M_k}_{\infty} \leq (\sup_{1\leq i \leq k} \norm{M_i}_{\infty})^k = (1+\sqrt{2})^k.
\end{align*}
In particular, the absolute value of every coefficient of $w$ is bounded by $(1+\sqrt{2})^k$. Hence, 
\[ \abs{\text{tr}(w)} \leq 2(1+\sqrt{2})^k.\]
Since we have that the length of the geodesic corresponding to $w$ is expressed as: 
\[ l_{\gamma}(w) = 2\mathrm{Re}\bigg[\arccosh \bigg(\frac{\textrm{trace}([w])}{2}\bigg)\bigg] = 2 \log\bigg(\left\lvert\frac{\text{tr}(w)}{2} + \sqrt{\Big(\frac{\text{tr}(w)}{2}\Big)^2-1}\right\rvert \bigg),\]
using the inequality above, we obtain that: 
\[ l_{\gamma}(w) \leq  k \cdot 2\log\Big(\frac{3}{2}(1+\sqrt{2})\Big). \]
Hence, if we have that $l_{\gamma}(w) > l_{i-1}$ for all $w\in W$, we can bound the probability above by: 
\[\mathbb{P}\bigg[ G_{Y_n} \ \text{contains no essential cycles of length } \ \leq \left \lfloor \frac{l_{i-1}}{2\log(\frac{3}{2}(1+\sqrt{2}))} \right\rfloor \bigg],\]
where an essential cycle in $G_{Y_n}$ refers to a cycle representing a curve in the manifold that is non-homotopic to a cusp or a point. As a note, the word essential often refers to curves that are also non-homotopic to a boundary component, but in our case we allow this condition.

If we denote by $\tau_{n}$ the minimum length of an essential cycle in $G_{Y_n}$, this is equivalent to:
\[\mathbb{P}\bigg[ \tau_{n} > \left \lfloor \frac{l_{i-1}}{2\log(\frac{3}{2}(1+\sqrt{2}))} \right\rfloor \bigg].\]
Now, to get the uniform bound on this, we use a version of the following result by McKay-Wormald-Wysocka:
\begin{corol}[\cite{Mckay-Wormald}, Corollary 1] \label{graph}
For $(d-1)^{2g-1} = o(n)$, the probability that a random $d$-regular graph has girth greater than $g\geq 3$ is 
\[\exp\bigg(- \sum_{r=3}^{g} \frac{(d-1)^r}{2r} + o(1)\bigg), \]
as $n\rightarrow \infty$.
\end{corol}

This follows from a more general result (see \cite[Theorem 1]{Mckay-Wormald}). The result is very close to what we need, but cannot be applied as is. Note that $\tau_{n}$ does not consider cycles that correspond to parabolic elements in the manifold $Y_n$. However, these are included in the result above. Hence, to get a bound for our case, we need to take out a factor corresponding to the possible number of parabolic elements appearing as cycles in the graph. 

One of the steps of the proof of Corollary \ref{graph} is to compute the ratio between the set of $d$-regular graph on $n$ vertices with fixed numbers $m_1, \ldots, m_t$ of cycles of certain lengths $c_1, \ldots, c_t$, and the set of those with at most $R_1, \ldots, R_t$ cycles of these lengths $c_1, \ldots, c_t$, where $m_i \leq R_i$, $i=1, \ldots, t$ and the $R_i$ are growing in $n$. For this, they use the switching method. More precisely, they count the average number of ways of applying a (forward and backward) switching to a graph to get from one set to the other.

By construction of the words $w$ one has that:
\[ \frac{\#\{\text{words $w\in W$ of $|w|= r$ corresponding to parabolic elements}\}}{\#\{\text{words $w\in W$ of $|w|= r$} \}} = \frac{1}{3^r}.\]

Indeed, since there are 9 matrices to choose from at each step, we can create $9^r$ possible words of length $r>0$. On the other hand, there exists only one class of words of length $r$ representing the parabolic elements. 
Indeed, let $w_1$, $w_2$ be two words of length $r > 0$ representing two parabolic elements, and suppose that $w_2 \notin [w_1]$.
Without lost of generality, we can take the word $w_1 = S^{r}$ as a representative of this class. 
Now, since $w_2$ also describes a parabolic element, it fixes an ideal vertex of an octahedron in the octahedral complex. Therefore, it admits as a representative for its class the rth-power of some parabolic letter fixing a vertex, that is, there is a word $w\in [w_2]$ that is of the form: 
\[ w = S^{r}, \quad w = (R\theta)^{r}, \text{ or} \quad w = (L\theta^2)^{r}.\]
Since $w_2 \notin [w_1]$, $w \neq S^{r}$. On the other hand, by looking at the relations defining the equivalence class of words explained in \cite{me}, one can check that the other two possibilities are equivalent to the first word. Hence $w\in [w_1]$, which implies that $w_1$ and $w_2$ define the same equivalence class. From the representative $w=S^r$, one can see that the cardinal of the class is $3^r$.


Hence, we deduce that the average number of cycles of length $r>0$ in $G_{Y_n}$ corresponding to essential curves is at most $\big(1-\frac{1}{3^r}\big)$ the average number of cycles of length $r$ in the graph. This needs to be taken into account in the proof of \cite[Theorem 1]{Mckay-Wormald} when counting the average ways of applying a backward switching, to make sure that the new created cycles are essential. With that in mind, together with some small changes in notation, the rest of the argument of \cite[Theorem 1]{Mckay-Wormald} follows step by step. We record the statement for our particular case as it will be used later on. 

\begin{corol} \label{graphme}
For $3^{2g-1} = o(n)$, the probability that $G_{Y_n}$ has no essential cycle of combinatorial length smaller or equal than $g\geq 3$ is less or equal than 
\[\exp\bigg(- \sum_{r=3}^{g} \frac{3^r}{2r}\Big(1-\frac{1}{3^r}\Big)+ o(1)\bigg), \]
as $n\rightarrow \infty$.
\end{corol}

We use this to bound the probability above. One observation is that $\left \lfloor \frac{l_{i-1}}{2\log(\frac{3}{2}(1+\sqrt{2}))} \right\rfloor$ is only greater than 3 from some length $l_k$ on. However, the number of terms for which this value is less than 3 is finite. Hence, for these values of $l_i$, we can bound the term (\ref{sysZn}) by:
\begin{align*}
    \mathbb{P}\big[A_{n}^{i}\big]\cdot l_i \leq \mathbb{P}\bigg[ \tau_{n} > \left \lfloor \frac{l_{i-1}}{2\log(\frac{3}{2}(1+\sqrt{2}))} \right\rfloor \bigg]l_k \leq l_k.
\end{align*}

On the other hand, we note that Corollary \ref{graphme} can be directly applied only when $g = \left \lfloor \frac{l_{i-1}}{2\log(\frac{3}{2}(1+\sqrt{2}))} \right\rfloor \leq (\frac{1}{2} + o(1))\log_3(n)$. For those values of $l_i \geq l_{k+1}$, we have: 
\begin{align*}
    \mathbb{P}\big[A_{n}^{i}\big]\cdot l_i &\leq \mathbb{P}\bigg[ \tau_{n} > \left \lfloor \frac{l_{i-1}}{2\log(\frac{3}{2}(1+\sqrt{2}))} \right\rfloor \bigg]\cdot l_i \\ 
    &\leq \exp\bigg(- \sum_{r=3}^{\lfloor \frac{l_{i-1}}{2\log(\frac{3}{2}(1+\sqrt{2}))} \rfloor} \frac{3^r}{2r}\Big(1-\frac{1}{3^r}\Big)+ K\bigg)\cdot l_i \\
    &= \exp\bigg(- \sum_{r=3}^{\left \lfloor \frac{l_{i-1}}{2\log(\frac{3}{2}(1+\sqrt{2}))} \right\rfloor} \frac{3^r+K'}{2r}\bigg)\cdot l_i \\
    &\leq \exp\bigg(-\frac{3^{\left \lfloor \frac{l_{i-1}}{2\log(\frac{3}{2}(1+\sqrt{2}))} \right\rfloor}+K'}{2{\left \lfloor \frac{l_{i-1}}{2\log(\frac{3}{2}(1+\sqrt{2}))} \right\rfloor}}\bigg) \cdot l_i,
\end{align*}
for some $K, K'>0$ independent of $n$.

Now, for all $l_i$ such that the previous condition is no longer satisfied, we use the following.
We refer here to an octahedral hyperbolic manifold a hyperbolic manifold built following the same procedure as $Y_n$, but in which the gluing is deterministic.
\begin{lemma}
    Let $X_n$ be an octahedral hyperbolic manifold made of $n$ octahedra. Then, its dual graph $G_{X_n}$ always has an essential cycle of combinatorial length  $\leq 4\lceil \log_2(\frac{n+1}{4}) \rceil + 1$.
\end{lemma}
\begin{proof}
We pick any vertex $v$ of $G_{X_n}$, and consider the four vertices neighbours to $v$. 
We argue by contradiction: suppose that there doesn't exist any essential cycle starting at any of these four vertices of length up to $k=4\lceil \log_2(\frac{n+1}{4}) \rceil+1$.

Now, we consider the paths starting at any of these four vertices that are of the form:
\[ w = w_1 \cdots w_{t}, \quad \text{with} \quad w_i\in \{SR\theta, SL\theta^2\}, \ \text{for } i=1,\ldots,t, \  t\in\mathbb{N}.\] 
Note that $SR\theta$ and $SL\theta^2$ are both two-letter words (we don't count $\theta$ or $\theta^2$ as letters) that correspond to essential paths, and so their contatenation also form an essential path.
Hence, all the paths described by the previous words are essential. 

Since, up to combintorial length $k$, these paths don't form an essential cycle, 
we have that the number of octahedra they go through after $t$ steps is $4\cdot 2^t$, for any $t\leq \frac{k-1}{4}$.
Equivalently, if $B(v, 2t)$ denotes a ball in $G_{X_n}$ of radius $2t\leq \frac{k-1}{2}$ around $v$, the number of vertices in $B(v, 2t)$ forming these paths is $4\cdot 2^{t}$.

However, if we take $t = \frac{k-1}{4}= \lceil \log_2(\frac{n+1}{4}) \rceil$, the previous fact tells us that the number of vertices forming these paths would be:
\[ 4\cdot 2^{\lceil \log_2(\frac{n+1}{4}) \rceil} > n,\] 
where $n$ is the total number of vertices. This gives us a contradiction, implying that there is at least two paths $w$, $w'$ of this form that have the same endpoint. 

Now, if $w$ and $w'$ have the same starting vertex, then their contatenation $w\cdot\bar{w}'$ -where $\bar{w}'$ denotes the backwards word of $w'$- forms a cycle.
Since the backwards word $\bar{w}$ of an essential path $w$ is an essential path, and both $w$ and $w'$ are essential ones, we obtain that $w\cdot\bar{w}'$ is an essential cycle of length at most $ 4\lceil \log_2(\frac{n+1}{4}) \rceil$.
On the other hand, if $w$ and $w'$ have different starting vertex, we know that both paths are connected to $v$. 
Thus, the path given by $w\cdot\bar{w}'\cdot \tilde{w}$, where $\tilde{w}\in \{S, R\theta, L\theta^2\}$ forms a cycle.
Since adding a letter to an essential path gives us an essential path, we conclude that $w\cdot\bar{w}'\cdot \tilde{w}$ forms an essential cycle of length at most $4\lceil \log_2(\frac{n+1}{4}) \rceil + 1$.
\end{proof}

Using this lemma, it is clear that for the $l_i$ such that $g = \left \lfloor \frac{l_{i-1}}{2\log(\frac{3}{2}(1+\sqrt{2}))} \right\rfloor \geq 4\lceil \log_2(\frac{n+1}{4}) \rceil + 1$, 
\[ \mathbb{P}\bigg[ \tau_{n} > \left \lfloor \frac{l_{i-1}}{2\log(\frac{3}{2}(1+\sqrt{2}))} \right\rfloor \bigg] = 0.\]

Finally, for the $l_i$'s such that $(\frac{1}{2} + o(1))\log_3(n) \leq \left \lfloor \frac{l_{i-1}}{2\log(\frac{3}{2}(1+\sqrt{2}))} \right\rfloor \leq 4\lceil \log_2(\frac{n+1}{4}) \rceil + 1$, we observe that the ratio: 
\[ \frac{\left \lfloor \frac{l_{i-1}}{2\log(\frac{3}{2}(1+\sqrt{2}))} \right\rfloor}{(\frac{1}{2} + o(1))\log_3(n)} \leq 13 + o(1) \leq C \quad \Rightarrow \quad \frac{\left \lfloor \frac{l_{i-1}}{2\log(\frac{3}{2}(1+\sqrt{2}))} \right\rfloor}{C} \leq \bigg( \frac{1}{2} + o(1)\bigg)\log_3(n).\]

Therefore, using Corollary \ref{graphme}, we obtain:
\begin{align*}
    \mathbb{P}\big[A_{n}^{i}\big]\cdot l_i &\leq \mathbb{P}\bigg[ \tau_{n} > \left \lfloor \frac{l_{i-1}}{2\log(\frac{3}{2}(1+\sqrt{2}))} \right\rfloor \bigg] l_i \\
    &\leq \mathbb{P}\bigg[ \tau_{n} > \frac{\left \lfloor \frac{l_{i-1}}{2\log(\frac{3}{2}(1+\sqrt{2}))} \right\rfloor}{C} \bigg] l_i\\
    &\leq \exp\bigg(-\frac{3^{\frac{\left \lfloor \frac{l_{i-1}}{2\log(\frac{3}{2}(1+\sqrt{2}))} \right\rfloor}{C}} +K''}{\frac{2}{C}{\left \lfloor \frac{l_{i-1}}{2\log(\frac{3}{2}(1+\sqrt{2}))} \right\rfloor}}\bigg) \cdot l_i,
\end{align*}
for some $C,K''>0$ independent of $n$. 

All in all, we get, for all values of $l_i$, the following upper-bound: 
\begin{lemma}  \label{upperbound_An} We have, for all $i\geq 1$:
    \[\mathbb{P}\big[A_{n}^{i}\big] \leq \exp\bigg(-\frac{3^{\frac{\left \lfloor \frac{l_{i-1}}{2\log(\frac{3}{2}(1+\sqrt{2}))} \right\rfloor}{C}} + K}{\frac{2}{C}{\left \lfloor \frac{l_{i-1}}{2\log(\frac{3}{2}(1+\sqrt{2}))} \right\rfloor}}\bigg), \]
for some $C,K>0$ independent of $n$. 
\end{lemma}

To complete the argument, we need now a lower and an upper bound on the lengths. For that, we have the following lemma: 
\begin{lemma}
    Let $l_k$, $k\geq1$, be the $k^{th}$-entry of $\{l_i\}_{i\geq1}$, the ordered set of all translation lengths coming from (classes of) words $[w] \in \mathcal{W}$. Then, for $k$ large enough, we have that:
    \[ K_1 \cdot \log(k) < l_k < K_2\cdot \log(k+3),\]
    for some $0<K_1 < \frac{1}{2}$ and $K_2\geq2$.
\end{lemma}
\begin{proof}
    We start with the upper bound. Consider the words of the form $w_k=S^{k+1}R\theta$, for $k\geq1$. They correspond to hyperbolic elements, so the translation lengths related to them are strictly positive. More precisely, one can compute that tr$(S^{k+1}R\theta) = k+3$, for any $k\geq1$. Hence, the translation lengths of the geodesics $\gamma$ corresponding their equivalence classes $[w_k]$ are given by:
    \[ l_k' = l_{\gamma}(w_k) =  2\arccosh\bigg(\frac{k+3}{2} \bigg) < 2\log(k+3).\]
    Since the list of lengths $\{l_k'\}_{k\geq1}$ derived from them form a subset of $\{l_i\}_{i\geq1}$, we obtain that: 
    \[l_k \leq l_k' < 2\log(k+3), \ \text{for all } k\geq1.\]

    Now, the argument for the lower bound relies on the following observation: we have that the group $\Gamma$ generated by the 9 matrices coming from (\ref{matrices}) is a subgroup of PSL$(2,\mathbb{Z}[i])$. This is a lattice of PSL$(2,\mathbb{C})$, so by the Prime geodesic theorem for hyperbolic manifolds \cite[Theorem 5.1]{SarnakPrime}, we know that the number of primitive closed geodesics of length up to some number $L$ in $\mathbb{H}^3/\text{PSL}(2,\mathbb{Z}[i])$ is asymptotic to:
    \[\# \{ [\gamma] \in \text{PSL}(2,\mathbb{Z}[i]) \ \text{primitive }: \ l(\gamma)< L \} \sim \frac{e^{2L}}{2L}.  \]

    That implies in particular that there are, asymptotically, at most exponentially many translation lengths up to $L$ in the length spectrum of this manifold. 
    
    Now, since the set $\{l_i\}_{i\geq1}$ is a subset of it -as it contains only the lengths coming form the matrices in $\Gamma$-, this tells us also that there are at most  exponentially many lengths up to $L$ in $\{l_i\}_{i\geq1}$. Let $\{l_1, \cdots, l_k\}$ be this set of translations lengths smaller than $L$. Then, the previous condition translates into:
\[k< e^{2l_k}\]
for $k$ large enough. From this we deduce that $l_k > K_1 \cdot \log(k)$, for some $0<K_1 < \frac{1}{2}$.
\end{proof}
All in all, summing over $i\geq1$, and denoting by $B =\log(\frac{3}{2}(1+\sqrt{2}))$, we obtain:
\begin{align*}
    \sum_{i=1}^{\infty} \mathbb{P}\big[A_{n}^{i}\big]\cdot l_i &= \sum_{i=1}^{k} \mathbb{P}\big[A_{n}^{i}\big]\cdot l_i + \sum_{i=k+1}^{\infty} \mathbb{P}\big[A_{n}^{i}\big]\cdot l_i \\
    &\leq kl_k + 2\sum_{i=k+1}^{\infty} \exp\bigg(- \sum_{r=3}^{\frac{1}{C}\left \lfloor \frac{l_{i-1}}{2B} \right\rfloor} \frac{3^r+K}{2r}\bigg)\cdot \log(i+3) \\
    &\leq kl_k + 2\sum_{i=k+1}^{\infty} \exp\bigg(-\frac{3^{\frac{1}{C}\left \lfloor \frac{l_{i-1}}{2B} \right\rfloor}+K}{\frac{2}{C}{\left \lfloor \frac{l_{i-1}}{2B} \right\rfloor}}\bigg) \cdot \log(i+3) \\ 
    &\leq kl_k + 2\sum_{i=k+1}^{\infty} \exp\bigg(-\frac{3^{\frac{1}{C}\lfloor \frac{\log(i-1)}{4B} \rfloor}+K}{\frac{2}{C}{\lfloor \frac{\log(i+2)}{B} \rfloor}}\bigg)\cdot \log(i+3) \\
    &\leq kl_k + 2\sum_{i=k+1}^{\infty} \exp\big(-2^{\lfloor \frac{\log(i-1)}{2B} \rfloor}\big)\cdot \log(i+3),
\end{align*}
where the latter is a convergent sum. Therefore, we can apply the dominated convergence theorem. This enables us to use Proposition \ref{limitZn}, and so obtain the expression (\ref{sysYn}) for the limit of $\mathbb{E}[\text{sys}(Y_n)]$.
The remaining step is then to see that a.a.s, this is also a valid expression for $\mathbb{E}[\text{sys}(M_n)]$. We approach this in the next section. 

\section{The systole for $M_n$} \label{systoleforMn}

The goal of this section is to prove that the contribution to the expected value of a set of possible ``bad'' manifolds $B_n$ arised from the compactification process is asymptotically negligible. In this way, we can conclude that the expected value of the systole computed in the previous section holds a.a.s for the manifolds $M_n$.

Recall that the manifolds $M_n$ are obtained from a Dehn filling procedure on the manifolds $Y_n$. 
This compactification process is done in three steps. The first one deals with the "small" cusps, that is, cusps made of few octahedra around them, and uses Andreev's theorem \cite{Andreev} to control the change in geometry. Then, the "medium" and "large" cusps are treated in two separate steps. In these cases, the main tool that assures enough control is a result of Futer-Purcell-Schleimer \cite[Theorem 9.30]{Futer_Purcell}. The complete argument can be found in \cite[Section 4]{me}. We recall here, nonetheless, the notation appearing in this paper.

\vspace*{0.2cm}
\begin{center}
\begin{tabular}{@{} *5l @{}} \toprule
\emph{Cusps of $Y_n$} & \emph{Description}  \\\midrule
 Small   & Of combinatorial length up to $\frac{1}{8}\log_3(n)$  \\ 
 Medium  & Of combinatorial length between $\frac{1}{8}\log_3(n)$ and $n^{1/4}$ \\ 
 Large   & Of combinatorial length bigger than $n^{1/4}$\\\bottomrule
 \hline
\end{tabular}
\end{center}

Under these definitions, we consider the following manifolds and their parts.

\vspace*{0.3cm}
\begin{tabular}{@{} *5l @{}}    \toprule
\emph{Models} & \emph{Description}  \\\midrule
 $Y_n$ & Non-compact hyperbolic 3-manifolds with totally geodesic boundary \\ 
 & obtained from a gluing of octahedra, and conditioned on not having  \\ 
 & loops or bigons in its dual graph\\\midrule
 $K_n$ & Manifold obtained from $Y_n$ by filling the small cusps\\ 
 $DK_n$ & Double of $K_n$ \\\midrule
 $M_n$ & Manifold obtained from $Y_n$ by filling the medium and large cusps \\ 
& homeomorphic to the $M_n$ described in Section \ref{probmodel} \\ 
 $DM_n$ & Double of $M_n$ \\\bottomrule
 \hline
\end{tabular}

\vspace{0.3cm}
Finally, note that $Y_n \subset K_n \subset M_n$, and let $\phi:Y_n \rightarrow M_n$ denote the inclusion map between these manifolds. This is the map we will refer to from now on, so the notation $\phi$ will be often omitted. 

Let us start by defining the set $B_n$. Informally, this set is formed either by those manifolds whose geometry gets distorted in the compactification process, or by the ones whose topological construction yields a degenerated systole. More precisely, this translates into the following subsets: 
\[B_n = B^{(1)}_n \cup B^{(2)}_n \cup B^{(3)}_n, \]
where: 
\begin{itemize}[leftmargin=*]
\setlength\itemsep{1em}
\item $B^{(1)}_n = \bigg \{  w\in\Omega_n : \begin{array}{c} \exists   \text{ a closed geodesic } \gamma\in M_n(w), \ l(\gamma) <  C\log(\log(n))  \\ \text{ s.t.} \ \forall \gamma' \text{ closed geodesic in } Y_n(w), \text{ s.t. }  \phi(\gamma') \text{ is homotopic to } \gamma: \\  \frac{l(\gamma)}{l(\gamma')} \notin [1-\epsilon, 1 + \epsilon], \ \text{for some } \epsilon>0 \end{array} \bigg \}$.
\item $B^{(2)}_n = \bigg \{ w\in\Omega_n : \begin{array}{c}  \exists   \text{ a closed geodesic } \gamma\in Y_n(w), \ l(\gamma) <  C\log(\log(n)) \\ \text{s.t.} \ \gamma  \ \text{becomes homotopically trivial in } M_n(w) \end{array} \bigg \}$.
\item $B^{(3)}_n = \{ w\in\Omega_n : \ \text{sys}(M_n(w)) \geq C\log(\log(n)) \}$.
\end{itemize}

Here $0<C<\frac{1}{20}$ is a fixed constant. 
Now, we can express the expected value of the systole for $M_n$ in a general way as follows:
\begin{align*}
    \mathbb{E}[\text{sys}(M_n)] &= \sum_{w\in\Omega_n} \mathbb{P}[w] \ \text{sys}(M_n(w)) \\
    &= \sum_{w\in\Omega_n \setminus B_n} \mathbb{P}[w] \ \text{sys}(M_n(w)) + \sum_{w \in B_n} \mathbb{P}[w] \ \text{sys}(M_n(w)) = E_n\textsuperscript{(1)}
     + E_n\textsuperscript{(2)}.
\end{align*}

\begin{remark}
    Note that there is a positive probability that the manifold $M_n$ is not hyperbolic, even if it's small (Theorem \ref{hyperbolic}). Hence, for these elements $w\in\Omega_n$, we set sys$(M_n(w))=0$.
\end{remark}
Thus, we want to prove that for this set of "bad" manifolds $B_n$, the following happens:

\begin{prop} \label{badset}
    \[\lim_{n \rightarrow \infty} E_n\textsuperscript{(2)} = \lim_{n \rightarrow \infty} \sum_{w \in B_n} \mathbb{P}[w] \ \text{sys}(M_n(w)) = 0.\]
\end{prop}

For this, we show that the limit of the sum under these sets $B^{(1)}_n, B^{(2)}_n$ and $B^{(3)}_n$ vanishes. We separate the proof into four different lemmas, the first studying the term sys$(M_n(w))$, and the rest the sum for each subset $B^{(1)}_n, B^{(2)}_n$ and $B^{(3)}_n$.

\begin{lemma}\label{sysbound} Let $w\in \Omega_n$. Then,
    \[ \text{sys}(M_n(w)) = O(\log(n)).\]
\end{lemma}

\begin{proof}
 By \cite{Bram_Jean}, we know that the boundary of $M_n(w)$ is a random closed hyperbolic surface $S_n(w)$ of genus $g\geq2$. By the Gauss-Bonnet theorem, its area is given by: 
\[ \text{area}(S_n(w)) = -2\pi\chi = 4\pi(g-1).\]
Now, this surface is built out of $4n$ triangles. Hence, by a simple Euler characteristic computation, we see that its genus has to be less than $n$.
With this, using the inequality for the systole given by the area growth \cite[Lemma 5.2.1]{Buserbook}, we get: 
\[ \text{sys}(S_n(w)) \leq 2\log(\text{area}(S_n(w))) + K \leq 2\log(4\pi(n-1)) = O(\log(n)).\]
Since the curves lying in $S_n(w)$ are part of the length spectrum of $M_n$, this is in turn an upper bound for $\text{sys}(M_n(w))$, that is, 
\[ \text{sys}(M_n(w)) \leq \text{sys}(S_n(w)) = O(\log(n)). \]
\end{proof}

\begin{lemma} \label{badset1}
    \[\lim_{n \rightarrow \infty} \sum_{w \in B^{(1)}_n} \mathbb{P}[w] \ \text{sys}(M_n(w)) = 0.\]
\end{lemma}

\begin{proof}

We bound the probability $\mathbb{P}[B^{(1)}_n]$.
As mentioned before, the control on the geometry -and therefore on the lengths of the curves- when doing the Dehn filling of the cusps relies on \cite{Andreev} and \cite[Theorem 9.30]{Futer_Purcell}. 

From \cite{Andreev}, one has bilipschitz equivalences between the thick parts of $Y_n$ and $K_n$. 
However, these bilipschitz constants may not be sufficiently small, or may accumulate if the cusps are very close to each other.
On the other hand, Theorem 9.30 from \cite{Futer_Purcell} gives bilipschitz equivalences between the thick parts of $K_n$ and $M_n$, provided that the total normalized length of the cusps $L$ satisfies: 
\begin{equation} \label{normalizedlength}
    L^2 \geq \frac{2\pi \cdot 6771\cosh{(0.6\delta + 0.1475)}^5}{ \delta^5} + 11.7.
\end{equation}
We observe that this won't be verified if $M_n$ has many large cusps, or if medium cusps are incident to each other (that is, they share an octahedron). Moreover, as before, there are some geometric conditions that, even if the theorem is applicable, may cause the bilitpschitz constants to degenerate.
Indeed, this might occur if geodesics enter the $\delta$-thin parts of the manifold $M_n$ -where the theorem doesn't give any control. 

Hence, since we can only assure that the length comparison is good enough when avoiding these cases, we redefine $B^{(1)}_n$ to be the set of manifolds for which any of the above occurs, that is, in which:
\begin{itemize}
    \item Small and medium cusps are incident.
    \item $Y_n(w)$ has many large cusps.
    \item Geodesics of length $\leq C\log(\log(n)$ enter the $\delta$-thin parts of the manifold $M_n(w)$, for some small $\delta(n)>0$.  
    \item There exists a closed geodesic $\leq C\log(\log(n)$ in $M_n(w)$ such that every preimage in $Y_n(w)$ goes into octahedra incident to small cusps.
\end{itemize} 
Then, outside this set, using the same arguments as in \cite[Proposition 4.1]{me}, we can conclude that the lengths pre and post compactification are comparable. In fact, the proof that follows gives an effective version of \cite[Proposition 4.1]{me} for curves of lengths up to $C\log(\log(n))$, as opposed to curves of uniformly bounded length.

Let's study the first case, that is, the probability that small and medium cusps are incident. Let $I_{c}$ denote the number of pairs of small or medium incident cusps. By \cite[Claim 1]{me}, we know that the expected number pairs of intersecting cusps of lengths exactly $k, l \leq C = o(n^{1/3})$ is $o(n^{-2/3})$. Thus, summing over all possible values of $k$ and $l$ -that go up to $o(n^{1/4})$ by definition of medium cusps-, gives that $\mathbb{E}[I_{c}] = o(n^{-1/6})$. 
Using Markov's inequality, we obtain, then:
\[ \mathbb{P}[I_{c} \geq 1] \leq \frac{\mathbb{E}[I_{c}]}{1} = o\bigg(\frac{1}{n^{1/6}}\bigg). \]

We analyze the next point: if $Y_n(w)$ has many large cusps. We have, by \cite[Theorem 2.4 (a)]{Bram_Jean}, that the expected number of cusps is $\frac{1}{2}\log(n) + O(1)$. So, by denoting as $C_l$ the number of large cusps in $Y_n(w)$, and applying Markov's inequality here again, we obtain: 
\[ \mathbb{P}[C_l \geq Kn^{1/4}, \ K\in(0,1) ] \leq \frac{\mathbb{E}[C_l]}{Kn^{1/4}} = O\bigg(\frac{\log(n)}{n^{1/4}}\bigg). \]

We deal now with the third case: if $\gamma$ enters the "very thin" part of the manifold $M_n$. Here, we can suppose that the previous bullet points don't occur a.a.s.
Also, note that it is enough to study this event in the thin parts of the compactified medium and large cusps, as we will deal with the components corresponding to the small cusps in the next case.  

We will show that closed geodesics of the length we consider don't enter the thin part at all. 
For this, we consider the double of the manifold resulting from the compactification of the small cusps, $DK_n$. 
This new manifold has medium and large cusps, that are then filled with solid cylinders. 
Then, we consider Margulis tubes of roughly the same area, so that the final manifold models the geometry of the hyperbolic metric in $DM_n$. 

Now, let $\gamma$ be a closed geodesic in the compactified manifold $DM_n$, lying only in one copy of $M_n$, and let $\delta = \frac{1}{\log(n)^{1/10}}$. Consider also the Margulis tubes $T_{r(\delta)}(\alpha)$ of radius $r(\delta)>0$ around a core curve $\alpha$, which contains the $\delta$-thin part of $DM_n$ around this core geodesic. Suppose that $\gamma$ enters the $\delta$-thin part of that manifold. 
That means, then, that it also enters another nested Margulis tube $T_{r(\epsilon)}(\alpha)$, for some fixed $\epsilon>\delta$ but small enough so that the $\epsilon$-thin part of the manifold is indeed still isometric to a standard Margulis tube. On the other hand, since $\gamma$ is a geodesic lying in one copy of $M_n$, we know that it cannot be entirely in the $\epsilon$-thin part. Therefore, its length needs to be at least twice the distance between the boundaries of the two tubes. 

Now, we would like to use the bounds on this distance given by Futer-Purcell-Schleimer in \cite[Theorem 1.1]{DistMargulistube}. For that, we need to check that the length of the core curve $\alpha$ we're considering is less than $\delta$. Recall that the Margulis tube around $\alpha$ corresponded, pre-compactification, to a cusp  neighbourhood around a medium or large cusp. Since these have total normalized length $L \geq \sqrt{\frac{1}{8}\log_3(n)}$, we get, by \cite[Corollary 6.13]{Futer_Purcell}, that the length of the core curve $\alpha$ is bounded by: 
\[ l(\alpha) < \frac{2\pi}{L^2 - 28.78} \leq \frac{2\pi}{\frac{1}{8}\log_3(n) - 28.78} < \frac{16\pi}{\log_3(n)}, \]
which is indeed less that $\delta$ when $n$ is large enough. Therefore, using now \cite[Theorem 1.1]{DistMargulistube}, we obtain that the distance between the boundary torii is bounded below by: 
\begin{align*}
    d(\partial T_{r(\delta)}(\alpha), \partial T_{r(\epsilon)}(\alpha)) &\geq \arccosh\bigg(\frac{\epsilon}{\sqrt{7.256 \delta}}\bigg)-0.0424 \\
    &> \arccosh\bigg(\frac{\epsilon}{\sqrt{7.256 \frac{1}{\log(n)^{1/10}}}}\bigg) \\ 
    &> \log\bigg(\frac{\epsilon\log(n)^{1/20}}{\sqrt{7.256}}\bigg)\\
    &> \log\bigg(\frac{\epsilon}{\sqrt{7.256}}\bigg) + \frac{1}{20}\log(\log(n)).
\end{align*}

This implies, then, that the length of $\gamma$ would need to be strictly bigger than $\frac{1}{10}\log(\log(n))$ to be able to enter the $\delta$-thin part around $\alpha$. However, we are considering geodesics of length less than $C\log(\log(n))$, where $C<\frac{1}{20}$. Therefore, this yields that, for $n$ big enough, $\gamma$ doesn't enter the $\delta$-thin parts corresponding to filled medium and large cusps.

A last remark is that it is enough to study this case for this value of $\delta$, that is, outside the $\delta$-thin part, the length of the curve is already controlled. Indeed, even if $\delta$ is tending to 0 as $n\rightarrow\infty$, the condition (\ref{normalizedlength}) on the total normalized length is still satisfied. Therefore \cite[Theorem 9.30]{Futer_Purcell} applies, and gives a bilipschitz equivalence between the $\delta$-thick parts of both manifolds, with bilipschitz constant tending to 1 as $n\rightarrow\infty$.

Finally, we study the last case. We suppose here again that the previous cases don't occur a.a.s. 
So, let $\gamma$ be some closed geodesic in $M_n$, and $\tilde{\gamma}$ a preimage in $Y_n$, and suppose that $\tilde{\gamma}$ enters into an octahedron incident to a small cusp. 
Then, we have that the cycles in the dual graph $G_{Y_n}$ corresponding to this curve and the parabolic element that goes around the cusp intersect. 
On the other hand, since the $\delta$-thick parts of $M_n$ and $K_n$ are bilipschitz with bilipschitz constant tending to 1, using the isometry from \cite{Andreev},
we can deduce that the length of the part of $\tilde{\gamma}$ lying outside the octahedra incident to the small cusp is of length less than $l(\gamma)<C\log(\log(n))$.

Hence, using now \cite[Proposition 3.4]{me}, we have that the part of the cycle in $G_{Y_n}$ corresponding to the part of $\tilde{\gamma}$ that is outside the octahedra incident to the cusp, has length bounded above by $\log(n)^C < \frac{1}{8}\log_3(n)$ for $n$ large enough. 
Therefore, we have found two cycles in $G_{Y_n}$ that lie inside a ball of diameter $\leq \frac{1}{8}\log_{3}(n)+ \log(n)^C$ from some common vertex. 
That would imply then that $G_{Y_n}$ is $l$-tangled for $l\leq\frac{1}{8}\log_3(n)$, and $n$ large enough. We recall that a multigraph $G$ is said to be 
$l$\textit{-tangled}, for $l>0$, if there exists some neighbourhood of radius $l>0$ in $G$ containing more than one cycle. Otherwise, we say that $G$ is $l$\textit{-tangle-free}.

Nevertheless, Lemma 9 from \cite{Tangle-free} tell us that random 4-regular graphs -so in particular $G_{Y_n}$- are $l$-tangle free, for $l>0$ not too large. More precisely, we have:
\[ \mathbb{P} \bigg[\ \textrm{$G_{Y_n}$ is $\frac{1}{8}\log_3(n)$-tangled}\ \bigg] = O\Big(\frac{3^{\frac{1}{2}\log_3(n)}}{n}\Big) \approx O\Big(\frac{1}{n^{1/2}}\Big).\]

A similar argument works if the preimage $\tilde{\gamma}$ enters into octahedra incident to different small cusps (which could happen if the small cusps are close to each other).
This would imply, then, that in the dual graph $G_{Y_n}$ we can find two cycles of combinatorial length $<\frac{1}{8}\log_3(n)$ at distance $<C\log(\log(n))$, yielding that $G_{Y_n}$ is $l$-tangled for $l\leq\frac{1}{8}\log_3(n)$. 
However, as shown above, the probability that this happens tends to zero as $n\rightarrow\infty$.

All together, we have that: 
\begin{align*}
    \mathbb{P}[B_n^{(1)}] &\leq \mathbb{P} [\ \textrm{$G_{Y_n}$ is $\frac{1}{8}\log(n)$-tangled}\ ] + \mathbb{P}[C_l > Cn^{1/4}] + \mathbb{P}[I_{c} \geq 1] \\
    &\leq O\Big(\frac{1}{n^{1/2}}\Big) + O\bigg(\frac{\log(n)}{n^{1/4}}\bigg) + o\bigg(\frac{1}{n^{1/6}}\bigg) \\ 
    &\leq O\Big(\frac{1}{n^{1/6}}\Big).
\end{align*}

Therefore, 
$B^{(1)}_n$
\[ \sum_{w \in B^{(1)}_n} \mathbb{P}[w] \ sys(M_n(w)) \leq \mathbb{P}[B^{(1)}_n]  \max_{w\in B^{(1)}_n} \{\text{sys}(M_n(w))\} \leq O\Big(\frac{1}{n^{1/6}}\cdot \log(n)\Big),\]
which tends to 0 as $n \rightarrow \infty$.

\end{proof}

\begin{lemma} \label{badset2}
    \[\lim_{n \rightarrow \infty} \sum_{w \in B^{(2)}_n} \mathbb{P}[w] \ \text{sys}(M_n(w)) = 0.\]
\end{lemma}

\begin{proof}
Note that for a short closed geodesic to be homotopically trivial after the compactification, it needs to go around at least one cusp in $Y_n$. Since the geodesic is of length $<C\log(\log(n))$, this needs to be a small cusp.

Now, if we look at the paths in the dual graph of $Y_n$ that this closed geodesic and the horocycle corresponding to the cusp do, we see that they are incident. 
This would imply, then, that $G_{Y_n}$ is $l$-tangled, for $l\leq\frac{1}{8}\log_3(n)$. However, again by \cite[Lemma 9]{Tangle-free}, we get that: 
\[ \mathbb{P}[B^{(2)}_n] \leq \mathbb{P}[\ \textrm{$G_{Y_n}$ is $\frac{1}{8}\log_3(n)$-tangled}\ ] = O\Big(\frac{1}{n^{1/2}}\Big).\]

Therefore, 
\[ \sum_{w \in B^{(2)}_n} \mathbb{P}[w] \ \text{sys}(M_n(w))  \leq O\Big(\frac{1}{n^{1/2}}\cdot \log(n)\Big) \rightarrow 0, \ \text{ as } n\rightarrow\infty.\]
\end{proof}

\begin{lemma} \label{badset3}
    \[\lim_{n \rightarrow \infty} \sum_{w \in B^{(3)}_n} \mathbb{P}[w] \ \text{sys}(M_n(w)) = 0.\]
\end{lemma}

\begin{proof}
For simplicity, here we can suppose that the length of short geodesics don't change when doing the Dehn filling, that is, that $w\notin B^{(1)}_n$ and $w\notin B^{(2)}_n$. Like this, the set $B^{(3)}_n$ can be also defined as:
\[ B^{(3)}_n = \{w\in\Omega_n : \text{sys}(Y_n(w)) > C'\log(\log(n)), \ \text{for} \ C' \in (0,1) \}. \]
This condition, in turn, can be translated into a condition on the length of paths in $G_{Y_n}$. More precisely, if the translation length of all closed geodesics in $Y_n$ is larger than $C'\log(\log(n))$, this implies that all corresponding closed paths have combinatorial length larger than $\frac{C'\log(\log(n))}{2\log(\frac{3}{2}(1+\sqrt{2}))} \geq \lfloor C''\log(\log(n)) \rfloor$, for $C'' \in (0, 1)$. 
Hence, the probability of the event $B^{(3)}_n$ is bounded by: 
\[\mathbb{P}[ w\in\Omega_n : G_{Y_n} \ \text{contains no essential cycles of lengths} \ \in \{3,\ldots, \lfloor C''\log(\log(n)) \rfloor\}] .\]
For this, we use again Corollary \ref{graphme}. This gives us: 
\begin{align*}
    \mathbb{P}[B^{(3)}_n] &\leq \exp\bigg(- \sum_{r=3}^{\lfloor C''\log(\log(n)) \rfloor\}} \frac{3^r}{2r}\bigg(1-\frac{1}{3^r}\bigg) + o(1)\bigg) \\
    & \leq K\exp\bigg(- \frac{3^{\lfloor C''\log(\log(n)) \rfloor}-1}{2\lfloor C''\log(\log(n)) \rfloor}\bigg) \quad \text{for some} \ K>0, \\
    & \leq O\bigg(\exp\big(-\frac{3^{\lfloor \log(\log(n)) \rfloor}}{\lfloor \log(\log(n)) \rfloor}\big) \bigg).
\end{align*}
Therefore, 
\[ \sum_{w \in B^{(3)}_n} \mathbb{P}[w] \ \text{sys}(M_n(w))  \leq O\Bigg(\frac{1}{e^{\frac{3^{\lfloor \log(\log(n)) \rfloor}}{\lfloor \log(\log(n)) \rfloor}}}\cdot \log(n)\Bigg),\]
which goes to 0 as $n \rightarrow \infty$.
\end{proof}

With all this, we are ready to prove Proposition \ref{badset}.
\begin{proof}[Proof of Proposition \ref{badset}]
Having studied the three cases, and using the systole bound from Lemma \ref{sysbound}, we obtain: 
\begin{align*}
    \sum_{w \in B_n} \mathbb{P}[w] \ \text{sys}(M_n(w)) &\leq \mathbb{P}[B_n]  \max_{w\in B_n} \{\text{sys}(M_n(w))\} \\ 
    &\leq (\mathbb{P}[B^{(1)}_n] + \mathbb{P}[B^{(2)}_n] + \mathbb{P}[B^{(3)}_n]) \max_{w\in B_n} \{\text{sys}(M_n(w))\} \\
    &\leq O\bigg(\frac{1}{n^{1/6}}\bigg) O\big(\log(n)\big),
\end{align*}
which tends to 0 as $n \rightarrow \infty$. 
\end{proof}

Now, this, together with Proposition \ref{limitZn}, enables to prove what we aimed for: 

\begin{reptheorem}{sysMn}
Let $\{l_i\}_{i\geq1}$ be the ordered set of all possible translation lengths coming from (classes of) words $[w] \in \mathcal{W}$. Then,
\[ \lim_{n\rightarrow \infty} \E[\text{sys}(M_n)] = \sum_{i= 1}^{\infty} \bigg( \prod_{\scriptscriptstyle[w]\in \mathcal{W}_{l_{i-1}}} \exp \bigg(\frac{\abs{[w]}}{2\abs{w}3^{\abs{w}}} \bigg)  \bigg) \bigg(1 - \prod_{\scriptscriptstyle [w]\in \mathcal{W}_{l_i} \setminus \mathcal{W}_{l_{i-1}}} \exp\bigg(\frac{\abs{[w]}}{2\abs{w}3^{\abs{w}}} \bigg)  \bigg)\cdot l_i.\]
\end{reptheorem}

\begin{proof}
    Using Proposition \ref{limitZn}, we proved in Section \ref{convergence} that the right hand side of the equality is a valid expression for the limit of the expected systole of $Y_n$. 

    On the other hand, Proposition \ref{badset} implies that, as $n \rightarrow \infty$, 
    \[\mathbb{E}(\text{sys}(M_n)) = \sum_{w\in\Omega_n \setminus B_n} \mathbb{P}[w] \ \text{sys}(M_n(w)). \]
    Since $B_n$ was exactly the set of manifolds for which the compactification process could degenerate the length of their curves, for all $w\in\Omega_n \setminus B_n$, we have that, $\forall \epsilon>0$: 
    \[\frac{1}{1+\epsilon}\lim_{n \rightarrow \infty} \mathbb{E}_n(\text{sys}(Y_n)) \leq \lim_{n \rightarrow \infty} \mathbb{E}_n(\text{sys}(M_n)) \leq (1+\epsilon)\lim_{n \rightarrow \infty} \mathbb{E}_n(\text{sys}(Y_n)).\]
    
    Therefore, all combined, we obtain the expression we are looking for. 
\end{proof}

\section{A numerical value} \label{numericalvalue}

 Since we have a formula for the limit of $\mathbb{E}(\text{sys}(M_n))$, we can try to compute a numerical value of it. The problem is, that the list of ordered lengths $l_i$ is hard to determine, and the program for computing the sets $\mathcal{W}_{l_i} \setminus \mathcal{W}_{l_{i-1}}$ for all $i\geq1$ is computationally very slow. 
 
 Indeed, even though the lengths can be computed using formula (\ref{l}), this equality depends on the trace of some class of words, which corresponds to a complex number. Hence, we cannot order lengths by trace, as these don't have a natural ordering. This fact differentiates it from the two dimensional case \cite{Bram1}, making the computation much harder. 

 One could consider, then, taking the $w$-distance (defined just after as \ref{distword}) to get the ordered list. 
 However, that doesn't completely work either, as it is not true that the translation length increases whenever this $w$-distance does so. 
 Another natural parameter for this is the combinatorial length of the (classes) of words. In this case, it does exists a coarse comparison between word length and geometric length \cite[Proposition 3.4]{me} which would enable us, a priori, to obtain this ordered list of lengths. However, the bound that this gives is too big to make it computationally feasible. 
 These obstacles also make it less evident to get a complete list of words of lengths less than $l_i$, for every $i\geq1$.

 On top of it, the complexity of the computation to check whether two words belong to the same equivalent class grows exponentially on the length of the word. Thus, computing the classes of words belonging to $\mathcal{W}_{l_i} \setminus \mathcal{W}_{l_{i-1}}$ is numerically doable only when $i$ is very small. 
 
 Hence, to get an approximated value for the limit, we do the following: we compute the first terms of the sum, for which the lengths and the sets $\mathcal{W}_{l_i} \setminus \mathcal{W}_{l_{i-1}}$ can be determined, and then we give an upper-bound for the rest of the sum. 
 To overcome the aforementioned constraints related to the ordered list of lengths, here we mix different techniques - using the $w$-distance in some cases, and the structure of the alphabet matrices in others- that allow us to compute more efficiently an ordered list of sufficient lengths to obtain a good approximation. 

 In order to simplify the formulas, we define some notation. For all $i\geq1$, let: 
\[ p_i = \lim_{n\rightarrow \infty}\mathbb{P}[A_n^i] = \bigg( \prod_{\scriptscriptstyle[w]\in \mathcal{W}_{l_{i-1}}} \exp \bigg(\frac{\abs{[w]}}{2\abs{w}3^{\abs{w}}} \bigg)  \bigg) \bigg(1 - \prod_{\scriptscriptstyle[w]\in \mathcal{W}_{l_i} \setminus \mathcal{W}_{l_{i-1}}} \exp\bigg(\frac{\abs{[w]}}{2\abs{w}3^{\abs{w}}} \bigg)  \bigg).\]

Then, we can write: 
\[ \lim_{n \rightarrow \infty} \mathbb{E}(\text{sys}(M_n)) = \sum_{i=1}^{k} p_i\cdot l_i + \sum_{i=k+1}^{\infty} p_i\cdot l_i = S_c + S_e, \]
where $S_c$ represents the computable part of the sum, and $S_e$ the error term. We study them separately in the next two subsections.

\subsection{The program for $S_c$}

To compute the finite sum $S_{c} = \sum_{i=1}^{k}p_i\cdot l_i$, we need to know the first $k$ values of $\{l_i\}_{i\geq1}$, and the (classes of) words that correspond to each of these lengths.


For that, we start by computing all words of translation length less than some number $D>0$. As mentioned before, this is not completely straightforward. To do so, we will (partly) use the $w$-distance, which we define as follows. 

Let $P$ be the plane determined by the triple $\{0,i,\infty\}$, and $w\in W$. Then, the $w$-distance, denoted by $d(w)$, is the distance in the upper half-space of $\mathbb{H}^3$ between the planes $P$ and $w(P)$. This is given by: 
\begin{align}\label{distword}
    d(w) &= \min_{\substack{(x_1,y_1)\in P \\ (x_2,y_2)\in w(P)}} \{d((x_1, y_1),(x_2, y_2))\} \\ &= \min_{\substack{(x_1,y_1)\in P \\ (x_2,y_2)\in w(P)}} \bigg\{ \arccosh\bigg(1+\frac{(x_2 - x_1)^2 + (y_2 - y_1)^2}{2y_1y_2}\bigg) \bigg\}.
\end{align}
Thus, we first compute the list of words of $d(w)<D$. This can be easily done as the $w$-distance increases whenever the combinatorial length of the word does. The reason is that the planed spanned by the larger word is contained in the half-space spanned by the smaller one. Indeed, if $H_2$ denotes the subset of $\mathbb{H}^3$ containing all points with positive first component, we have that, for any of the 9 isometries $M\in\{S\theta^{i}, R\theta^{i}, L\theta^{i}, \ i=0,1,2\}$, 
 \[M(H_2) \subset H_2.\]
From this, one infers that for any isometry $w$ product of the previous matrices, we get:
\[w(M(H_2)) \subset w(H_2).\]

Figure \ref{Planes} shows an example (for more details, see \cite[Proposition 3.4]{me}). This tells us, in particular, that the axis of any isometry $w$ intersects both $P$ and $w(P)$, as well as all the planes spanned by all shorter words of $w$. From this, and given the definition of the distance, we deduce that the $w$-distance corresponding to some word $w$ is always less or equal that the translation length of the curve related to it, that is,
\[ d(w) \leq l_{\gamma}(w), \quad \text{for any word } w\in W.\]
As a consequence, we can conclude that the list $L$ of words of $d(w)<D$ contains all words of translation length less than D. Indeed, for any word $w$ such that $d(w)>D$, we would have $D<d(w)\leq l_{\gamma}(w)$. Hence, the remaining step to get the initial list is to filter the words in $L$ by computing their translation length.

\begin{figure}[H]
    \centering
    \begin{minipage}{0.45\textwidth}
        \centering
        \includegraphics[width=0.9\textwidth]{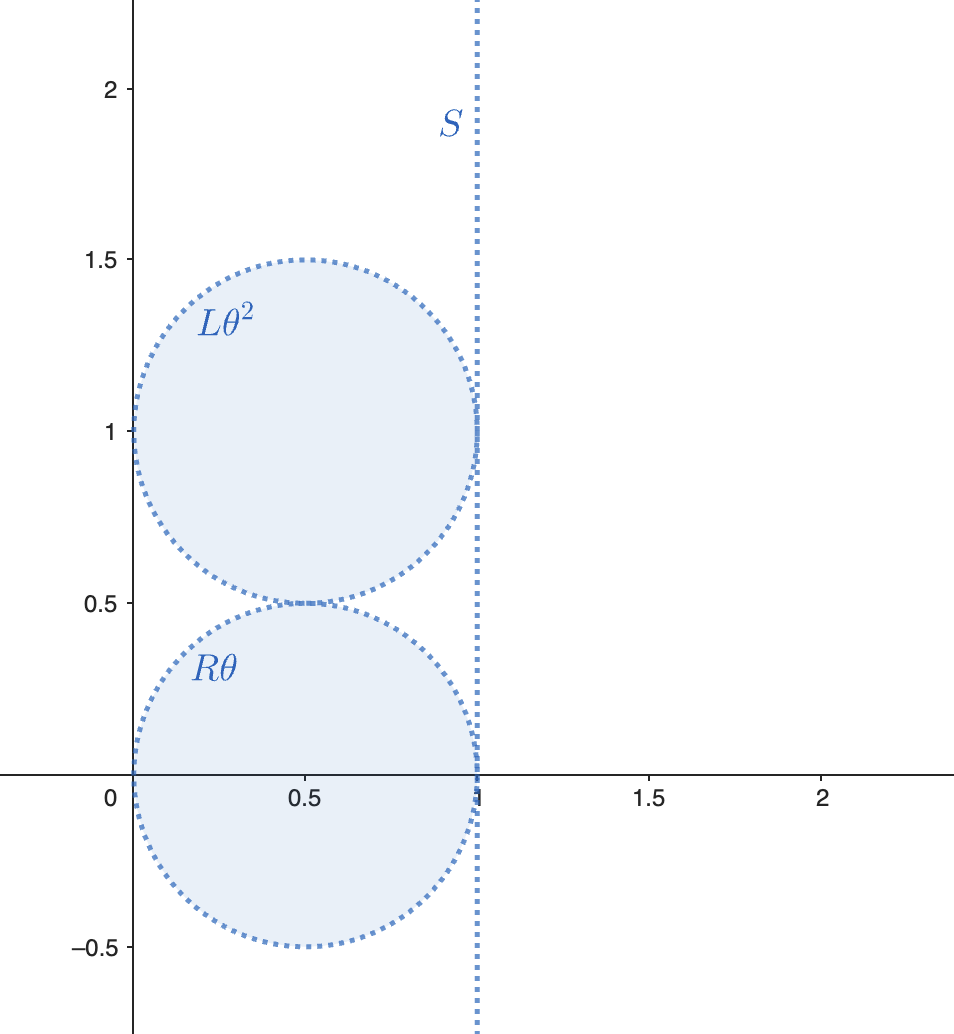} 
    \end{minipage}\hfill
    \begin{minipage}{0.45\textwidth}
        \centering
        \includegraphics[width=0.9\textwidth]{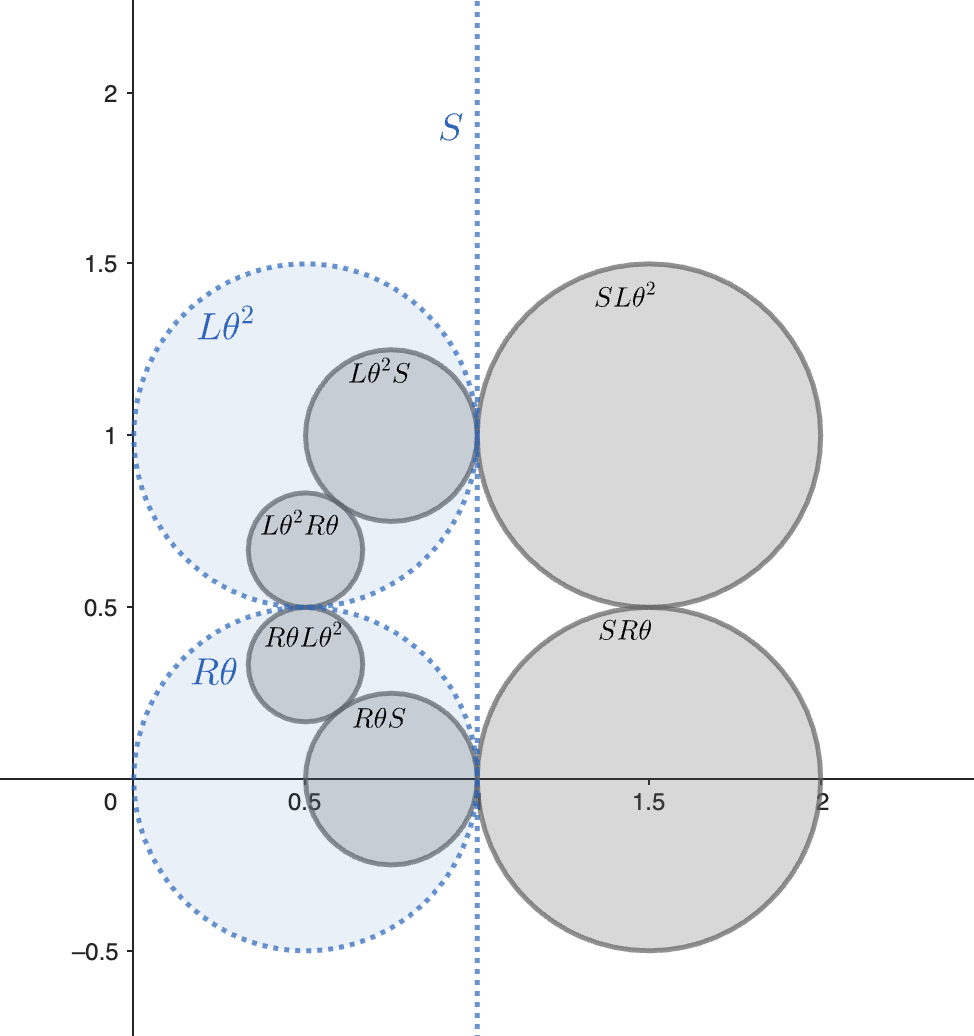}
    \end{minipage}
    \caption{View from "infinity" in the upper half-space model of $\mathbb{H}^3$: the circles in blue represent the hemispheres spanned by 1-letter words (represented by $S, R\theta, L\theta^2$), and the ones in grey all the hemispheres spanned by the 2-letter hyperbolic words of positive distance from $P$.}
    \label{Planes}
\end{figure}

This procedure, however, fails for certain words. More precisely, for those corresponding to parabolic elements with an extra twist when gluing the faces of the octahedra, for instance $w=SSS\theta$. These now correspond to hyperbolic elements, so their translation length is positive, but their $w$-distance is always zero. Hence, we need another way to see when we need to stop considering them. For that, we compute directly their translation length, and rely on the following lemma: 

\begin{lemma}
    Let $l_{\gamma}(w)$ denote the translation length of a closed curve $\gamma \in Y_n$ corresponding to a class of words $[w]\in \mathcal{W}$. Then, for all $k\geq3$, we have:
    \[l_{\gamma}(S^k\theta) < l_{\gamma}(S^{k+1}\theta)\]
    \[l_{\gamma}(S^k\theta^2) < l_{\gamma}(S^{k+1}\theta^2).\]
\end{lemma}

\begin{proof}
One can compute that the traces of words in $[S^k\theta]$ are equal to $\pm (ki+1)$, and hence those of $[S^{k+1}\theta]$ are $\pm ((k+1)i+1)$, for every $k\geq1$. Now, recall that:
\begin{align*}
    l_{\gamma}(S^k\theta) &= 2 \text{Re}\bigg[\arccosh\bigg(\frac{ki+1}{2}\bigg)\bigg] = 2 \text{Re}\bigg[\ln\bigg(\frac{ki+1}{2} + \sqrt{\Big( \frac{ki+1}{2}\Big)^2-1}\bigg) \bigg] \\
    &= 2 \log\bigg(\left\lvert\frac{ki+1}{2} + \sqrt{\Big(\frac{ki+1}{2}\Big)^2-1}\right\rvert \bigg).
\end{align*}
Thus, to see that the translation of these word increases, it is sufficient to see that the absolute value of this complex number $z$ inside the $\log$ does so. We know that $\abs{z} = \sqrt{\text{Re}(z)^2 + \text{Im}(z)^2}$, where: 
\[
    \text{Re}(z) = \frac{1}{2} + \Re\bigg(\sqrt{2ki - \frac{k^2+3}{4}}\bigg) = \frac{1}{2} + \bigg( \frac{1}{2}\sqrt[\leftroot{-2}\uproot{2}4]{k^4+70k^2+9} \cdot \sin\Big(\frac{1}{2}\arctan(\frac{8k}{k^2+3})\Big) \bigg) 
\] \[ \quad \text{Im}(z) = \frac{k}{2} + \Im\bigg(\sqrt{2ki - \frac{k^2+3}{4}}\bigg) = \frac{k}{2} + \bigg( \frac{1}{2}\sqrt[\leftroot{-2}\uproot{2}4]{k^4+70k^2+9} \cdot \cos\Big(\frac{1}{2}\arctan(\frac{8k}{k^2+3})\Big) \bigg). \]  
By computing their squares, and summing them, we get: 
\begin{align*}
    \abs{z}^2 &= \text{Re}(z)^2 + \text{Im}(z)^2 \\ 
    &= \frac{k^2+1}{4} + \frac{1}{4}\sqrt{k^4+70k^2+9} \\ &\quad +  \frac{1}{2} \sqrt[\leftroot{-2}\uproot{2}4]{k^4+70k^2+9}\bigg(\sin\Big(\frac{1}{2}\arctan(\frac{8k}{k^2+3})\Big) + k\cos\Big(\frac{1}{2}\arctan(\frac{8k}{k^2+3})\Big) \bigg).
\end{align*}
The first line of the expression is clearly increasing in $k$, for $k\geq3$. For the second line, consider $f(k) = \sin\Big(\frac{1}{2}\arctan(\frac{8k}{k^2+3})\Big) + k\cos\Big(\frac{1}{2}\arctan(\frac{8k}{k^2+3})\Big)$. 
Since $0<\sin(\frac{1}{2}\arctan(\frac{8k}{k^2+3})<0.55$ and $0.83<\cos(\frac{1}{2}\arctan(\frac{8k}{k^2+3})<1$ for all $k\geq3$, $f(k)$ is a positive function. Now, if we compute its derivative, we have that:
\[ f'(k) = \frac{4k(k^2-3)\sin(\frac{1}{2}\arctan(\frac{8k}{(k^2+3)})) + (k^4+66k^2+21)\cos(\frac{1}{2}\arctan(\frac{8k}{(k^2+3)}))}{k^4+70k^2+9} > 0.\]
As the product and composition of positive increasing functions is positive and increasing, we finally obtain that $\abs{z} = \sqrt{\text{Re}(z)^2 + \text{Im}(z)^2}$ is increasing in $k$, which is what we wanted. 

For the class $[S^k\theta^2]$, an analogous argument works: the traces of words in $[S^k\theta^2]$ are equal to $\pm (ki-1)$, and hence those of $[S^{k+1}\theta^2]$ are $\pm ((k+1)i-1)$, for every $k\geq1$. We get, then, the same expressions for the real an imaginary parts, so the rest follows from above.  

\end{proof}
This tells us, then, that once one of these words have translation length bigger than $D$, we can stop checking the larger ones of that same form. 

Joining these two procedures, we obtain the complete list of word of translation length less than $D$. This process is carried out by a SageMath program \cite{sagemath}, a link to which is provided at the end of the section. Here we present a pseudo code.

\begin{algorithm}[H]
\caption{Computes the list of words of translation length less than D}\label{main}
\begin{algorithmic}[1]
\Procedure{Trans. length}{$D$}
\State $\mathbf{Set}$ list $tocheck_{hyp}$ \Comment{of hyp. words that we need to check at each step} 
\State $\mathbf{Set}$ list $valid_{hyp}$ \Comment{of words from $tocheck_{hyp}$ that have $w$-distance $<$ D} 
\State $\mathbf{Set}$ list $tocheck_{par}$ \Comment{of parab. words with twists that we need to check at each step} 
\State $\mathbf{Set}$ list $valid_{par}$ \Comment{of words from $tocheck_{par}$ that have tr. length $<$ D} 
\State $\mathbf{Set}$ list $valid$ \Comment{of all words of tr. length $<$ D} 
\State $\mathbf{Set}$ list $trace_{w}$ \Comment{of traces of all words from $valid$} 
\\
\BState \emph{Initial case}:
\State $tocheck_{hyp} \gets \text{3-tuples of indices from} \ \{0,1,2\}$ 
\For{\text{each tuple in} $tocheck_{hyp}$} 
\If{tuple is a parabolic word}
\State $w2, w3 \gets \text{Matrices corresponding to tuple, with one and two twists}$
\State $l2 \gets \text{translation length corresponding to }w2$
\If{l2 \text{is less than} D}
\State \small SAVE PAR\normalsize($valid$, $valid_{par}$, $trace_w$, $tuple$, $w2$)
\EndIf
\State $l3 \gets \text{translation length corresponding to }w3$
\If{l3 \text{is less than} D}
\State \small SAVE PAR\normalsize($valid$, $valid_{par}$, $trace_w$, $tuple$, $w3$)
\EndIf
\Else 
\State $w \gets \text{Matrix corresponding to tuple}$
\State $d\gets \text{$w$-distance from} \ P=\{0, i, \infty\} \ \text{to} \ w(P)$
\If{d \text{is less than} D}
\State $valid_{hyp} \gets \text{tuple}$
\State $l \gets \text{translation length corresponding to }w$
\If{l \text{is less than} D}
\State \small SAVE\normalsize($valid$, $trace_w$, $tuple$, $w$)
\EndIf
\State $w2, w3 \gets \text{Matrices corresponding to tuple, with one and two twists}$
\State $l2 \gets \text{translation length corresponding to }w2$
\If{l2 \text{is less than} D}
\State \small SAVE\normalsize($valid$, $trace_w$, $tuple$, $w2$)
\EndIf
\State $l3 \gets \text{translation length corresponding to }w3$
\If{l3 \text{is less than} D}
\State \small SAVE\normalsize($valid$, $trace_w$, $tuple$, $w3$)
\EndIf
\EndIf
\EndIf
\EndFor
\EndProcedure
\algstore{myalg}
\end{algorithmic}
\end{algorithm}

\setcounter{algorithm}{0}
\begin{algorithm}[H]
\caption{Computes the list of words of translation length less than D}
\begin{algorithmic}[1]
\algrestore{myalg}
\Procedure{Trans. length}{$D$}
\\
\BState Iterative case:
\State n=3
\While{\text{lists} \ $valid_{hyp}$ \ \text{or} \ $valid_{par}$ \ \text{are non-empty}}
\State $tocheck_{hyp} \gets \text{empty the list}$
\State $tocheck_{hyp} \gets \text{(n+1)-tuples of indices from} \ valid_{hyp} \ \text{and} \ \{0,1,2\}$
\State $tocheck_{par} \gets \text{(n+1)-tuples of indices from} \ valid_{par}$ \ \text{by repeating 1st index}
\State $valid_{hyp} \gets \text{empty the list}$
\State $valid_{par} \gets \text{empty the list}$
\\
\For{\text{each tuple in} $tocheck_{hyp}$} 
\State $w \gets \text{Matrix corresponding to tuple}$
\State $d\gets \text{$w$-distance from} \ P=\{0, i, \infty\} \ \text{to} \ w(P)$
\If{d \text{is less than} D}
\State $valid_{hyp} \gets \text{tuple}$
\State $l \gets \text{translation length corresponding to }w$
\If{l \text{is less than} D}
\State \small SAVE\normalsize($valid$, $trace_w$, $tuple$, $w$)
\EndIf
\State $w2, w3 \gets \text{Matrices corresponding to tuple, with one and two twists}$
\State $l2 \gets \text{translation length corresponding to }w2$
\If{l2 \text{is less than} D}
\State \small SAVE\normalsize($valid$, $trace_w$, $tuple$, $w2$)
\EndIf
\State $l3 \gets \text{translation length corresponding to }w3$
\If{l3 \text{is less than} D}
\State \small SAVE\normalsize($valid$, $trace_w$, $tuple$, $w3$)
\EndIf
\EndIf
\EndFor
\\
\For{\text{each tuple in} $tocheck_{par}$} 
\State $w_p \gets \text{Matrix corresponding to tuple}$
\State $l_p \gets \text{translation length corresponding to} \ w_p$
\If{$l_p$  \text{is less than}  D}
\State \small SAVE PAR\normalsize($valid$, $valid_{par}$ $trace_w$, $tuple$, $w_p$) 
\EndIf
\EndFor
\State $n \gets n+1$
\EndWhile
\\
\State $length \gets \text{length of list} \ valid$
\State \Return $valid$, $length$, $trace_w$ \Comment{returns the list of all valid words, its length, and the list of their traces}

\EndProcedure
\end{algorithmic}
\end{algorithm}

\begin{algorithm}[H]
\caption{Arranges and saves the valid words coming from parabolics}\label{arrange2}
\begin{algorithmic}[2]
\Procedure{Save par}{$valid$, $valid_{par}$, $trace_w$, $tuple$, $w$}
\State $tc \gets \text{change indices of tuple to 9 matrix system}$
\State $valid_{par} \gets tc$ 
\State $valid \gets tc$ 
\State $trace_w \gets \text{trace of }w$
\EndProcedure
\end{algorithmic}
\end{algorithm}

\begin{algorithm}[H]
\caption{Arranges and saves the valid words}\label{arrange}
\begin{algorithmic}[2]
\Procedure{Save}{$valid$, $trace_w$, $tuple$, $w$}
\State $tc \gets \text{change indices of tuple to 9 matrix system}$
\State $valid \gets tc$ 
\State $trace_w \gets \text{trace of }w$
\EndProcedure
\end{algorithmic}
\end{algorithm}

As a note, we can think of this program in a more geometric way: when computing all words of $w$-distance less that $D$, we are constructing a polyhedron made of octahedra which contains all paths $w$ whose $w$-distance $d(w)$ with respect to some initial face P is less than $D$.

Once we have this list of words of translation length $<D$, we group them in classes of words -by analysing, for each one, the conditions that define the equivalence class-, and compute the cardinals of each class. On the other hand, we compute the complete list of translation lengths that appear up to the number $D$. These lists give us, then, all the information we need to compute $S_c$. 

Taking $D=4.6$, we obtain a list of 31 lengths, and a value of $S_c$:
\[ S_c = \sum_{i=1}^{31} p_i\cdot l_i = 2.56033312683887522062 \ldots \]

The entire SageMath code used for the computation is available at:
\url{https://github.com/annaroigsanchis/Systole-computation---Sage-code}.

\subsection{The error term $S_e$} 

In order to bound the sum $S_e$, we subdivide it into blocks, and bound each of these, that is,
\[ S_e = \sum_{i=32}^{\infty} p_i \cdot l_i = \sum_{i: l_i \in (l_{31}, \lceil l_{32} \rceil)} p_i \cdot l_i \ + \sum_{k=\lceil l_{32} \rceil}^{\infty} \sum_{i: l_i \in [k, k+1)} p_i \cdot l_i.\]

Observe that in these sums, the lengths $l_i$ can be bounded by $\lceil l_{32} \rceil$ and $k+1$ respectively, which are natural numbers. Thus, to get a sharper bound on the error term, then, we decrease the growth in the sum over $k$ by defining: 
\[ \tau(l_i)=2\cosh\Big(\frac{l_i}{2}\Big),\]
and re-writing the previous expression as: 
\[ S_e = \sum_{i: \tau(l_i) \in (\tau(l_{32}), \lceil \tau(l_{32}) \rceil)} p_i \cdot l_i \ + \sum_{k=\lceil \tau(l_{32}) \rceil}^{\infty} \sum_{i: \tau(l_i) \in [k, k+1)} p_i \cdot l_i \eqqcolon A + B.\]

The term A can be bounded by: 
\begin{align*}
    A &\leq 2\arccosh\Big(\frac{\lceil \tau(l_{32}) \rceil}{2} \Big) \lim_{n\rightarrow\infty} \mathbb{P}[\Z = 0, \ \forall \ [w]\in \mathcal{W} \ : \tau(l_{\gamma}(w)) < \tau(l_{32})] \\ 
    &= 2\arccosh\Big(\frac{\lceil \tau(l_{32}) \rceil}{2} \Big) \lim_{n\rightarrow\infty} \mathbb{P}[\Z = 0, \ \forall \ [w]\in \mathcal{W}_{l_{31}}] = 2.9220\cdot 10^{-16}.
\end{align*}
where the probability can found using the computation for $S_c$. 

Now, let's study B. As before, we have that: 
\[ B \leq \sum_{k=\lceil \tau(l_{32}) \rceil}^{\infty} 2\arccosh\Big(\frac{k+1}{2} \Big) \lim_{n\rightarrow\infty} \mathbb{P}[\Z = 0, \ \forall \ [w]\in \mathcal{W} \ : \tau(l_{\gamma}(w)) < k]. \]
Denote by $q_{[0,k)} = \mathbb{P}[\Z = 0, \ \forall \ [w]\in \mathcal{W} \ : \tau(l_{\gamma}(w)) < k]$. We can decompose $q_{[0,k)}$ as: 
\[ q_{[0,k)} = q_{[0,\tau(l_{32}))} \cdot q_{[\tau(l_{32}), k)}.\]
The first factor corresponds to the same probability as in A, so we can compute it. Like this, we get:
\[ B \leq e^{-\frac{112}{3}}\sum_{k=\lceil \tau(l_{32}) \rceil}^{\infty} 2\arccosh\Big(\frac{k+1}{2} \Big) \lim_{n\rightarrow\infty} \mathbb{P}[\Z = 0, \ \forall [w]\in \mathcal{W}: \ \tau(l_{\gamma}(w)) \in [\tau(l_{32}), k)]. \]
To get a sharper bound, we study the first term $B_1$ (when $k=\lceil \tau(l_{32}) \rceil$) separately. Using computational data, we have that the probability:
\[ \mathbb{P}[\Z = 0, \ \forall \ [w]\in \mathcal{W} \ : \tau(l_{\gamma}(w)) \in [\tau(l_{32}), \lceil \tau(l_{32}) \rceil)]\]
can be bounded above by:
\[ \mathbb{P}[\Z = 0, \ \text{for} \ [w] \in \Lambda],\]
where:
\[\Lambda = \bigg\{\begin{array}{c} [S^{10}\theta], [S^{10}\theta^2], [S^2(L\theta^2)^2 S\theta], [S^2(L\theta)^2 S\theta], [S^2 (L\theta^2)^2 R\theta], \\ 
\ [S^2(R\theta)^2 L\theta^2], [S^3(L\theta^2)^2R\theta^2], [S^3(R\theta)^2L\theta], [S^3(S\theta)^2R], [S^3(S\theta^2)^2L] 
\end{array}\bigg\}.\]

Each of these classes of words have $\lambda_{[w]}=\frac{1}{2}$, so by Theorem \ref{limitZn} we have that, as $n\rightarrow\infty$, this probability is exactly $e^{-5}$. Hence, the first term is bounded by: 
\[ B_1 \leq e^{-\frac{112}{3}}2\arccosh\Big(\frac{\lceil \tau(l_{32}) \rceil+1}{2} \Big)e^{-5} = e^{-\frac{112}{3}}2\arccosh(6)e^{-5} = 7.51077 \cdot 10^{-19}.\]
Finally, we deal with the remaining term, that is: 
\[ B_2 = e^{-\frac{112}{3}}\sum_{k=\lceil \tau(l_{32}) \rceil+1}^{\infty} 2\arccosh\Big(\frac{k+1}{2} \Big) \lim_{n\rightarrow\infty} \mathbb{P}[\Z = 0, \ \forall \ [w]\in \mathcal{W} \ : \tau(l_{\gamma}(w)) \in [\tau(l_{32}), k)].\]
Similarly as before, we argue that this probability above can be bounded by: 
\[\mathbb{P}[\Z = 0, \ \text{for} \ [w] \in \{ \Lambda \cup [S^{r-3}R\theta], \ r\in [\lceil \tau(l_{32}) \rceil, k]\} ].\]
Indeed, one can compute that $\tau(l([S^{k-3} R\theta])) = \text{tr}([S^{k-3} R\theta])=k-1$, which verifies the condition: $\tau(l_{\gamma}(w)) \in [\tau(l_{32}), k)$. On the other hand, it is not hard to see that $|[S^{k}R\theta]|=3^k 2k$. Therefore, we have that: 
\[\mathbb{P}[\Z = 0, \ \text{for} \ [w] \in \{ \Lambda \cup [S^{r-3}R\theta], \ r\in [\lceil \tau(l_{32}) \rceil, k]\}] \leq e^{-5} \cdot e^{\lceil \tau(l_{32}) \rceil - k}.\]
Hence, $B_2$ can be bounded by: 
\begin{align*}
    B_2 &\leq 2e^{-\frac{112}{3}}e^{-5}e^{\lceil \tau(l_{32}) \rceil} \sum_{k=\lceil \tau(l_{32}) \rceil+1}^{\infty} \arccosh\Big(\frac{k+1}{2} \Big) e^{-k} \\
    &= 2e^{-\frac{112}{3}}e^{6} \sum_{k=12}^{\infty} \arccosh\Big(\frac{k+1}{2} \Big) e^{-k} \simeq 1.24718 \cdot 10^{-18}.
\end{align*}
All together, we have that the error term $S_e$ is bounded by: 
\[ S_e = A + B \leq A + B_1 + B_2 \leq 2.95489 \cdot 10^{-16}.\]

Joining the values obtained for $S_c$ and $S_e$, we finally obtain: 
  
\begin{repprop}{sysvalue}
    We have:
    \[2.5603331268388752 \leq \lim_{n \rightarrow \infty} \mathbb{E}(\text{sys}(M_n)) \leq 2.5603331268388752 +  2.95489 \cdot 10^{-16}.\]  
\end{repprop}


\bibliographystyle{alpha}
\bibliography{bib}

\end{document}